\documentclass[a4paper,11pt]{amsart}
\usepackage{amsfonts,amssymb,amsmath,amsthm,abstract,color}
\usepackage[mathscr]{euscript}
\usepackage[ps,all,arc,rotate]{xy}
\usepackage[lmargin=1in,rmargin=1in,tmargin=1in,bmargin=1in]{geometry}
\usepackage{fancyhdr}
\usepackage{pb-diagram}
\usepackage{enumitem}
\usepackage{hyperref}

\newtheorem{prop}{Proposition}[section]
\newtheorem{lemma}[prop]{Lemma}
\newtheorem{thm}[prop]{Theorem}
\newtheorem{cor}[prop]{Corollary}
\newtheorem{conj}[prop]{Conjecture}

\theoremstyle{definition}
\newtheorem{defn}[prop]{Definition}

\newtheorem{rmk}[prop]{Remark}
\newtheorem{ex}[prop]{Example}

\DeclareMathOperator{\tr}{tr} 
\DeclareMathOperator{\rk}{rk}        
\DeclareMathOperator{\spec}{Spec} \DeclareMathOperator{\Sym}{Sym}  
\DeclareMathOperator{\MotCpl}{MotCpl}
\DeclareMathOperator{\coker}{Coker}  
\DeclareMathOperator{\id}{id}

\DeclareMathOperator{\diag}{diag}
\newcommand{\ra}{\rightarrow}      

\newcommand{\lra}{\longrightarrow}

\newcommand{\Gy}{\mathrm{Gy}}
\newcommand{\gm}{\mathrm{gm}}

\DeclareMathOperator{\Bun}{Bun}
\DeclareMathOperator*{\holim}{holim} 
\DeclareMathOperator*{\hocolim}{hocolim}
\DeclareMathOperator*{\colim}{colim}
\DeclareMathOperator{\cone}{Cone} 
\DeclareMathOperator{\Div}{Div} \DeclareMathOperator{\quot}{Quot}
\DeclareMathOperator{\Jac}{Jac} \DeclareMathOperator{\Sm}{Sm}
\DeclareMathOperator{\Hom}{Hom}
\DeclareMathOperator{\Spec}{Spec}
\DeclareMathOperator{\Pic}{Pic}
\DeclareMathOperator{\Ext}{Ext}
\DeclareMathOperator{\codim}{codim}
\DeclareMathOperator{\cdh}{cdh}
\DeclareMathOperator{\eff}{eff}
\DeclareMathOperator{\Nis}{Nis}
\DeclareMathOperator{\Gpds}{Gpds}
\DeclareMathOperator{\sSets}{sSets}
\DeclareMathOperator{\abgp}{Ab}
\DeclareMathOperator{\sing}{sing}

\DeclareMathOperator{\Aut}{Aut}
\DeclareMathOperator{\Var}{Var}
\DeclareMathOperator{\atlas}{atlas}
\DeclareMathOperator{\nerve}{nerve}
\DeclareMathOperator{\mono}{mono}
\DeclareMathOperator{\epi}{epi}

\newcommand{\et}{\mathrm{\acute{e}t}}

\def\cA{\mathcal A}
\def\cE{\mathcal E}\def\cF{\mathcal F}
\def\cK{\mathcal K}\def\cL{\mathcal L}
\def\cM{\mathcal M}\def\cN{\mathcal N}\def\cO{\mathcal O}\def\cP{\mathcal P}
\def\cT{\mathcal T}
\def\cU{\mathcal U}\def\cX{\mathcal X}
\def\cY{\mathcal Y}\def\cZ{\mathcal Z}

\def\AA{\mathbb A}
\def\FF{\mathbb F}\def\GG{\mathbb G}
\def\LL{\mathbb L}
\def\NN{\mathbb N}\def\PP{\mathbb P}
\def\QQ{\mathbb Q}\def\RR{\mathbb R}

\def\ZZ{\mathbb Z}

\def\fX{\mathfrak X}

 \def\GL{\mathrm{GL}} \def\SL{\mathrm{SL}}

\def\DM{\mathrm{DM}}  \def\DA{\mathrm{DA}}   
\def\CH{\mathrm{CH}}

\makeatletter
\def\@tocline#1#2#3#4#5#6#7{\relax
  \ifnum #1>\c@tocdepth 
  \else
    \par \addpenalty\@secpenalty\addvspace{#2}%
    \begingroup \hyphenpenalty\@M
    \@ifempty{#4}{%
      \@tempdima\csname r@tocindent\number#1\endcsname\relax
    }{%
      \@tempdima#4\relax
    }%
    \parindent\z@ \leftskip#3\relax \advance\leftskip\@tempdima\relax
    \rightskip\@pnumwidth plus4em \parfillskip-\@pnumwidth
    #5\leavevmode\hskip-\@tempdima
      \ifcase #1
       \or\or \hskip 1em \or \hskip 2em \else \hskip 3em \fi%
      #6\nobreak\relax
    \hfill\hbox to\@pnumwidth{\@tocpagenum{#7}}\par
    \nobreak
    \endgroup
  \fi}
\makeatother

\title{On the Voevodsky motive of the moduli stack of vector bundles on a curve}

\author{Victoria Hoskins and Simon Pepin Lehalleur}

\thanks{V.H. is supported by the Excellence Initiative of the DFG at the Freie Universit\"{a}t Berlin and by the SPP 1786. S.P.L. is supported by the Einstein Foundation, through the Einstein visiting Fellowship 0419745104 \lq Algebraic Entropy, Algebraic Cycles' of Professor V. Srinivas.}

\begin{document}

\maketitle

\begin{abstract}
We define and study the motive of the moduli stack of vector bundles of fixed rank and degree over a smooth projective curve in Voevodsky's category of motives. We prove that this motive can be written as a homotopy colimit of motives of smooth projective Quot schemes of torsion quotients of sums of line bundles on the curve.  When working with rational coefficients, we prove that the motive of the stack of bundles lies in the localising tensor subcategory generated by the motive of the curve, using Bia{\l}ynicki-Birula decompositions of these Quot schemes. We conjecture a formula for the motive of this stack, and we prove this conjecture modulo a conjecture on the intersection theory of the Quot schemes.
\end{abstract}

\tableofcontents

\section{Introduction}

Let $C$ be a smooth projective geometrically connected curve of genus $g$ over a field $k$. We denote the moduli stack of rank $n$, degree $d$ vector bundles on $C$  by $\Bun_{n,d}$; this is a smooth algebraic stack of dimension $n^{2}(g-1)$. The cohomology of $\Bun_{n,d}$ has been studied using a wide array of techniques, and together with Harder--Narasimhan stratifications, these results are used to study the cohomology of the moduli space  $\cN_{n,d}$ of semistable vector bundles for coprime $n$ and $d$. In this paper, we study the motive of $\Bun_{n,d}$ in the sense of Voevodsky.

Let us start with a chronological survey of the various results on the cohomology of $\Bun_{n,d}$. One of the first calculations was a stacky point count of $\Bun_{n,d}$ over a finite field $\FF_q$ due to Harder \cite{Harder}; the formula (\textit{cf.}\ Theorem \ref{thm harder}) is remarkably simple and involves the point count of the Jacobian of $C$ and a product of Zeta functions. These point counting methods enabled Harder and Narasimhan \cite{HN} to compute inductive formulae for the Betti numbers of moduli spaces $\cN_{n,d}$ for $n$ and $d$ coprime over the complex numbers via the Weil conjectures. Atiyah and Bott \cite{atiyah_bott} gave an entirely different approach to this computation by using a gauge theoretic construction of $\cN_{n,d}$ over the complex numbers to study its cohomology. Their calculation of the Betti cohomology of the classifying space of the gauge group describes $H^*(\Bun_{n,d},\QQ)$ as an algebra generated by the K\"{u}nneth components of the Chern classes $c_i(\cU) \in H^*(\Bun_{n,d} \times C,\QQ) \cong H^*(\Bun_{n,d},\QQ) \otimes H^*(C,\QQ)$ of the universal bundle $\cU \ra \Bun_{n,d} \times C$. In fact, these results generalise to the stack $\Bun_G$ of principal $G$-bundles over $C$ for a reductive group $G$ by \cite{hs}. Finally, let us mention the work of Gaitsgory-Lurie \cite{Gaitsgory_Lurie} on the computation of the $\ell$-adic cohomology of $\Bun_{G}$, which is inspired by the Atiyah-Bott method.

Another algebro-geometric approach to studying the $\ell$-adic cohomology of $\Bun_{n,d}$ and $\cN_{n,d}$ for coprime $n$ and $d$ was given by Bifet, Ghione and Letizia \cite{bgl}, using matrix divisors to rigidify and approximate the cohomology of $\Bun_{n,d}$; this approach is tailored to $G = \GL_n$ and does not naturally easily to other groups. Since their techniques are very geometric, they are amenable to being used in a range of different contexts, and the ideas in \cite{bgl} lie at the heart of this paper. The ideas of \cite{bgl} were also used by Behrend and Dhillon \cite{BD} to give a closed formula for the class of $\Bun_{n,d}$ in a dimensional completion of the Grothendieck ring of varieties (\textit{cf.}\ Theorem \ref{thm BD formula}). Furthermore, using the ideas in \cite{bgl}, the Chow motive of $\cN_{n,d}$ was studied by Del  Ba\~{n}o in \cite{Del_Bano_motives_moduli}, and the motivic cohomology of $\cN_{n,d}$ was considered in \cite{ADK}. 

Let $\DM(k,R)$ be the triangulated category of mixed motives over $k$ with coefficients in a commutative ring $R$, where if $p=\mathrm{char}(k)>0$, we assume that either $p$ is invertible in $R$ or that $k$ is perfect and admits resolution of singularities. Any $k$-variety $X$ has a motive $M(X)$ in $\DM(k,R)$ which is a refined cohomological invariant of $X$ that contains information about both the cohomology of $X$ and algebraic cycles on $X$. For an algebraic stack $\fX$ over $k$, it is not completely straightforward to define a motive $M(\fX)$ in $\DM(k,R)$, especially if $R$ is not assumed to be a $\QQ$-algebra. In $\S$\ref{sec motive stacks}, we define $M(\fX)$ for a class of smooth stacks which we call exhaustive stacks (\textit{cf.}\ Definition \ref{def exh seq}) by adapting ideas of Totaro \cite{totaro} and Morel--Voevodsky \cite{MVW} for quotient stacks to this more general setting. Informally, an algebraic stack $\fX$ is exhaustive if it can be well approximated by a sequence of schemes which occur as open substacks of vector bundles over increasingly large open substacks of $\fX$. In Appendix~\ref{app mot stacks}, we discuss alternative approaches for defining \'etale motives of stacks.

We then proceed to show that $\Bun_{n,d}$ is exhaustive and start computing its motive. To approximate $\Bun_{n,d}$ by schemes, we rigidify vector bundles using matrix divisors as in \cite{bgl,BD}. For an effective divisor $D$ on $C$, a matrix divisor of rank $n$, degree $d$ on $C$ with pole divisor $D$ is a injective homomorphism $\cE\hookrightarrow\cO_C(D)^{\oplus n}$ of coherent sheaves on $C$ with $\cE$ a rank $n$, degree $d$ locally free sheaf. The space $\Div_{n,d}(D)$ of matrix divisors of rank $n$, degree $d$ and pole divisor $D$ is a Quot scheme parametrising torsion quotients of the bundle $\cO_C(D)^{\oplus n}$ of a given degree; since we are looking at torsion quotients over a curve, $\Div_{n,d}(D)$ is a smooth projective variety. 

We can now state our first main result.

\begin{thm}\label{intro_main_thm}
For any effective divisor $D_0>0$, we have
\[
M(\Bun_{n,d})\simeq \hocolim_{l\in\NN} M(\Div_{n,d}(lD_0)).
\]
In particular, the motive $M(\Bun_{n,d})$ is pure, in the sense that it lies in the heart of the Chow weight structure on $\DM(k,R)$.
\end{thm}

The proof requires among other things to correct a codimension estimate in \cite{BD}; see Theorem \ref{prop exh seq} and Lemmas \ref{codim bgl correct} and \ref{coconut}.

To describe the motives of the varieties $\Div_{n,d}(lD_0)$, we use a $\GG_m$-action and a Bia{\l}ynicki-Birula decomposition \cite{BB_original} as in \cite{bgl}. The varieties $\Div_{n,d}(lD_0)$ come with natural actions of $\GL_n$, such that the morphisms $\Div_{n,d}(lD_0)\ra \Div_{n,d}((l+1)D_0)$ are equivariant. If we restrict the action to a generic one-parameter subgroup $\GG_m\subset \GL_n$, then the connected components of the fixed point locus can be identified with products of symmetric powers of $C$. By applying a motivic Bia{\l}ynicki-Birula decomposition \cite{Brosnan, Choudhury_Skowera, Karpenko}, we obtain the following result. 

\begin{thm}\label{intro_main_thm2}
Assume that $R$ is a $\QQ$-algebra; then $M(\Bun_{n,d})$ lies in the localising tensor triangulated category of $\DM(k,R)$ generated by $M(C)$. Hence, $M(\Bun_{n,d})$ is an abelian motive. 
\end{thm}

The assumption that $R$ is a $\QQ$-algebra is used to show that the motive of a symmetric power of $C$ is a direct factor of the motive of a power of $C$. For $C= \PP^1$, as symmetric products of $\PP^1$ are projective spaces, we deduce that $M(\Bun_{n,d})$ is a Tate motive for any coefficient ring $R$. 

We then conjecture the following formula for the motive of $\Bun_{n,d}$.

\begin{conj}\label{main conj}
Suppose that $C(k) \neq \emptyset$; then in $\DM(k,R)$, we have
\[M(\Bun_{n,d}) \simeq  M(\Jac(C)) \otimes M(B\GG_m) \otimes \bigotimes_{i=1}^{n-1} Z(C, R(i)[2i]). \]
where $Z(C,R(i)[2i]):=\bigoplus_{j=0}^{\infty} M(X^{(j)})\otimes R(ij)[2ij]$ denotes the motivic Zeta function.
\end{conj}

To prove this formula, one needs to understand the behaviour of the transition maps in the inductive system given in Theorem \ref{intro_main_thm} with respect to the motivic Bia{\l}ynicki-Birula decompositions. We formulate a conjecture (\textit{cf.}\ Conjecture \ref{thm2}) on the behaviour of these transitions maps with respect to these decompositions; this is equivalent to a conjecture concerning the intersection theory of the the smooth projective Quot schemes $\Div_{n,d}(D)$ (\textit{cf.}\ Conjecture \ref{conj int thy quot} and Remark \ref{rmk conj int thy implies conj motives}). In Theorem \ref{thm3}, we prove that Conjecture \ref{thm2} implies Conjecture \ref{main conj}; hence, it suffices to solve the conjecture on the intersection theory of these Quot schemes, which is ongoing work of the authors. 

The assumption that $C$ has a rational point is needed so that Abel-Jacobi maps from sufficiently large symmetric powers of $C$ to $\Jac(C)$ are projective bundles (\textit{cf.}\ Remark \ref{rmk litt}). In fact, these projective spaces then contribute to the motive of $B\GG_m$ (\textit{cf.}\ Example \ref{ex mot BGm}).

Finally let us state some evidence to support Conjecture \ref{main conj}, as well as some consequences. First, using Poincar\'{e} duality for smooth stacks (\textit{cf.}\ Proposition \ref{prop PD for stacks}), we deduce a formula for the compactly supported motive of $\Bun_{n,d}$ (\textit{cf.}\ Theorem \ref{compact supp}), which is better suited to comparisons with the results concerning the topology of $\Bun_{n,d}$ mentioned above. In $\S$\ref{sec compare}, we explain how this conjectural formula for $M^c(\Bun_{n,d})$ is compatible with the Behrend--Dhillon formula \cite{BD} by using a category of completed motives inspired by work of Zargar \cite{Zargar} (\textit{cf.}\ Lemma \ref{lemma comp BD}). In $\S$\ref{sec fixed det}, we deduce from Conjecture \ref{thm2} formulae for the motive (and compactly supported motive) of the stack $\Bun_{n,d}^L$ of vector bundles with fixed determinant $L$ and the stack $\Bun_{\SL_n}$ of principal $\SL_n$-bundles over $C$.

The structure of this paper is as follows: in $\S$\ref{sec motives}, we summarise the key properties of motives of schemes and define motives of smooth exhaustive stacks. In $\S$\ref{sec motive bun}, we prove Theorems \ref{intro_main_thm} and \ref{intro_main_thm2}, state Conjecture \ref{main conj}, and prove that this conjecture follows from a conjecture concerning the intersection theory of the smooth projective Quot schemes $\Div_{n,d}(D)$ (\textit{cf.}\ Theorem \ref{thm3}). In $\S$\ref{sec conseq and comp}, we deduce from Conjecture \ref{main conj} a formula for the compactly supported motive of $\Bun$, which we compare with previous formulae in $\S$\ref{sec compare} and we deduce formulae for the motives of $\Bun_{n,d}^L$ and $\Bun_{\SL_n}$ from Conjecture \ref{thm2} (\textit{cf.}\ Theorems \ref{fixed det} and \ref{SL-bundles}). Finally, in Appendix~\ref{app mot stacks}, we explain and compare alternative approaches for defining \'etale motives of stacks.\\

\noindent \textbf{Notation and conventions.}
Throughout all schemes and stacks are assumed to be defined over a fixed field $k$. By an algebraic stack, we mean a stack for the fppf topology with an atlas given by a representable, smooth surjective morphism from a locally of finite type scheme.\footnote{This definition is slightly different from the standard one where the atlas is allowed to be an algebraic space, and a condition on the diagonal is enforced; however this definition is sufficient for our purposes.} For a closed substack $\cY$ of an algebraic stack $\fX$, we define the codimension of $\cY$ in $\fX$ to be the codimension of $\cY\times_{\fX} U$ in $U$ for an atlas $U\ra \fX$; this is independent of the choice of atlas.

For $n \in \NN$ and a quasi-projective variety $X$, the symmetric group $\Sigma_n$ acts on $X^{\times n}$ and the quotient is representable by a quasi-projective variety $X^{(n)}$, the $n$-th symmetric power of $X$.\\

\noindent \textbf{Acknowledgements.} We thank Michael Gr\"{o}chenig, Jochen Heinloth, Marc Levine, Dragos Oprea, Rahul Pandaripande, and Alexander Schmitt for useful discussions.

\section{Motives of schemes and stacks}
\label{sec motives}
Let $k$ be a base field and $R$ be a commutative ring of coefficients. If the characteristic $p$ of $k$ is positive, then we assume either that $p$ is invertible in $R$ or that $k$ is perfect and admits the resolution of singularities by alterations. We let $\DM(k,R):= \DM^{\Nis}(k,R)$ denote Voevodsky's category of (Nisnevich) motives over $k$ with coefficients in $R$; this is a monoidal triangulated category. For a separated scheme $X$ of finite type over $k$, we can associate both a motive $M(X)\in \DM(k,R)$, which is covariantly functorial in $X$ and behaves like a homology theory, and a motive with compact supports $M^c(X)\in \DM(k,R)$, which is covariantly functorial for proper morphisms and behaves like a Borel-Moore homology theory.

Without going into the details of the construction, we recall that objects in $\DM(k,R)$ can be represented by motivic complexes; that is, objects in the category $\MotCpl(k,R)$ of (symmetric) $T$-spectra in complexes of Nisnevich sheaves with transfers\footnote{We recall that sheaves with transfers have additional contravariant functoriality for finite correspondences.} of $R$-modules on the category $\Sm_k$ of smooth $k$-schemes, where
\[T:=\coker(R^{\eff}_{\tr}(\Spec k)\rightarrow R^{\eff}_{\tr}(\GG_m))[-1]. \]
Here $R^{\eff}_{\tr}(X)$ denotes the sheaf of finite correspondences into $X$ with $R$-coefficients for $X\in \Sm_{k}$. We write $R_{\tr}(X)$ for the suspension spectrum $\Sigma^{\infty}_{T}R^{\eff}_{\tr}(X)$. A morphism in $\MotCpl(k,R)$ that becomes an isomorphism in $\DM(k,R)$ is called an $\AA^1$-weak equivalence (abbreviated in $\AA^1$-w.e.).

\begin{rmk}
The main results of this paper hold in the category $\DM^{\eff}(k,R)$ of effective motives. We refrained from writing everything in terms of $\DM^{\eff}(k,R)$ for two reasons.
  \begin{enumerate}[label={\upshape(\roman*)}]
  \item Under our assumptions on $k$ and $R$, the functor $\DM^{\eff}(k,R)\ra \DM(k,R)$ is fully faithful \cite{cancellation} \cite{suslin_imperfect}, so that results in $\DM^{\eff}(k,R)$ follow immediately from their stable counterparts.
  \item The motive with compact support $M^{c}(\fX)$ of an Artin stack, however it is defined, is almost never effective (see Section $\S$\ref{sec mot comp}).
  \end{enumerate}
\end{rmk}  

\subsection{Properties of motives of schemes}
The category $\DM(k,R)$ was originally constructed in \cite{VSF} and its deeper properties were established under the hypothesis that $k$ is perfect and satisfies resolution of singularities (with no assumption on $R$). They were extended to the case where $k$ is perfect by Kelly in \cite{kelly}, using Gabber's refinement of de Jong's results on alterations. Finally, the extension of scalars of $\DM$ from a field to its perfect closure was shown to be an equivalence in \cite[Proposition 8.1.(d)]{Cisinski_Deglise_cdh}. 

The motive $M(\Spec k):=R_{\tr}(\Spec k)$ of the point is the unit for the monoidal structure, and there are Tate motives $R(n)\in \DM(k,R)$ for all $n\in\ZZ$. For any motive $M$ and $n \in \ZZ$, we write $M(n):=M\otimes R(n)$, and we write $M\{n\} := M(n)[2n]$.

Let us list the main properties of motives that will be used in this paper. 
\begin{itemize}
\item (K\"unneth formula): for schemes $X$ and $Y$, we have  
\[M(X\times_k Y)\simeq M(X)\otimes M(Y) \quad \text{and} \quad M^c(X\times_k Y)\simeq M^c(X)\otimes M^c(Y).\]
\item ($\AA^1$-homotopy invariance): by construction of $\DM(k,R)$, for any Zariski-locally trivial affine bundle $Y\rightarrow X$ with fibre $\AA^r$, the following induced morphisms are isomorphisms
\[ M(Y)\rightarrow M(X) \quad \text{and} \quad M^c(Y)\{r\}\rightarrow M^c(X). \]
\item (Motives with and without compact supports): for a separated finite type scheme $X$, there is a morphism $M(X)\rightarrow M^c(X)$, which is an isomorphism if $X$ is proper.
\item (Projective bundle formula): for a vector bundle $E\rightarrow X$ of rank $r+1$, there are isomorphisms
\[M(\PP(E))\simeq \bigoplus_{i=0}^{r}M(X)\{i\} \quad \text{and} \quad
M^c(\PP(E))\simeq \bigoplus_{i=0}^{r}M^c(X)\{i\}.\]
\item (Gysin triangles): for a closed immersion $i:Z\ra X$ of codimension $c$ between smooth $k$-schemes, there is a functorial distinguished triangle
\[
  M(X - Z)\rightarrow M(X)\stackrel{\Gy(i)}{\rightarrow} M(Z)\{c\}\stackrel{+}{\rightarrow}.
\]
\item (Flat pullbacks for $M^c$): for a flat morphism $f:X\rightarrow Y$ of relative dimension $d$, there is a pullback morphism 
\[f^*:M^c(Y)\{d\}\rightarrow M^c(X).\]
\item (Localisation triangles): for a separated scheme $X$ of finite type and $Z$ any closed subscheme, there is a functorial distinguished triangle
\[
M^c(Z)\rightarrow M^c(X)\rightarrow M^c(X - Z)\stackrel{+}{\rightarrow}.
\]
\item (Internal homs and duals): the category $\DM(k,R)$ has internal homomorphisms, which can be used to define the dual of any motive $M \in \DM(k,R)$ as 
\[ M^\vee := \underline{\text{Hom}}(M,R(0)).\]
\item (Poincar\'{e} duality): for a smooth scheme $X$ of pure dimension $d$, there is an isomorphism
\[ M^c(X) \simeq M(X)^\vee \{ d \}.\]
\item\label{motalg cycles} (Algebraic cycles): for a smooth scheme $X$ (say of pure dimension $d$ for simplicity), a separated scheme $Y$ of finite type and $i\in \NN$, there is an isomorphism 
  \[ \CH_i(X\times Y)_R \simeq \Hom_{\DM}(M(X),M^c(Y)\{d-i\})\]
where $\CH_i$ denotes the Chow groups of cycles of dimension $i$. 
\item (Compact generators): $\DM(k,R)$ is compactly generated by $M(X)(-n)$ for $X \in \Sm_k$ and $n\in \NN$. Let $\DM_{\gm}(k,R)$ denote the triangulated subcategory consisting of compact objects\footnote{We recall that an object $M\in \DM(k,R)$ is compact if and only if for all families $(N_i)_{i \in I}\in \DM(k,R)^I$ indexed by a set $I$, the natural map $\bigoplus_{i \in I} \Hom(M,N_i)\ra \Hom(M,\bigoplus_{i\in I}N_i)$ is an isomorphism.}. For a separated finite type scheme $X$, both $M(X)$ and $M^c(X)$ are compact. 
\end{itemize}

\begin{rmk}
The category of \'etale motives $\DM^{\text{\'et}}(k,R)$, which is defined by replacing the Nisnevich topology with the \'{e}tale topology, does not capture the information about integral and torsion Chow groups; however, defining motives of stacks is technically simpler as we explain in Appendix~\ref{app mot stacks}. 
\end{rmk}  

\subsection{Homotopy (co)limits}\label{sec holim van}

As $\DM(k,R)$ is a compactly generated triangulated category, it admits arbitrary direct sums and arbitrary direct products \cite[Proposition 8.4.6]{neeman}; hence, one can define arbitrary homotopy colimits and homotopy limits for $\NN$-indexed systems in $\DM(k,R)$ using only the triangulated structure together with direct sums and products as follows.

\begin{defn}\label{defn_hocolim}
The homotopy colimit of an inductive system $F_*: \NN\rightarrow \DM(k,R)$ is 
\[
\hocolim_{n\in \NN} F_n:= \cone\left(\bigoplus_{i\in \NN} F_{i} \stackrel{\id-\sigma}{\longrightarrow} \bigoplus_{i\in \NN} F_i\right)
\]
and the homotopy limit of a projective system $G^*:\NN^{\text{op}}\rightarrow \DM(k,R)$ is 
\[
\holim_{n\in \NN} G^n:= \cone\left(\prod_{i\in \NN} G^i \stackrel{\id-\sigma}{\longrightarrow} \prod_{i\in \NN} G^i\right)[-1],
\]
where we write $\sigma$ for any of the maps $F_i\rightarrow F_{i+1}$ (resp.\ $G^{i+1}\rightarrow G^{i}$) in the diagram.

Note that, by construction, for any given choice of such a cone, there is a compatible system of maps, i.e. an element of $\lim_i \Hom(F_i,\hocolim F_*)$ (resp.\ $\lim_j \Hom(\holim G^*, G^j)$).
\end{defn}

\begin{rmk}
Using \cite[Lemma 1.7.1]{neeman}, it is easy to extend this definition to homotopy colimits indexed by filtered partially ordered sets $I$ such that there exists a cofinal embedding $\NN\ra I$. 
\end{rmk}

The reader unfamiliar with this definition should compare it with the definition of the limit and its derived functor $R^1\lim$ for $\NN$-indexed diagrams of abelian groups in \cite[Definition 3.5.1]{Weibel}. Since homotopy (co)limits are defined by the choice of a cone, they are only unique up to the a non-unique isomorphism. However, in the case where the $\NN$-indexed system actually comes from the underlying model category (that is, it can be realised as a system of $T$-spectra of complexes of Nisnevich sheaves with transfers and morphisms between them), then homotopy colimits can be realised in a simple, canonical way as follows.

\begin{lemma}\label{colim_hocolim}
Let $S_*: \NN\rightarrow\MotCpl(k,R)$ be an inductive system of motivic complexes. Then
  \[
\hocolim_{n}S_{n}\simeq \colim_{n}S_{n}
\]
with the colimit being computed in the abelian category $\MotCpl(k,R)$.
\end{lemma}
\begin{proof}
As $\MotCpl(k,R)$ is a Grothendieck abelian category, it has exact filtered colimits and so
  \[
0\ra \bigoplus_{n\geq 0} S_{n}\stackrel{\id-\sigma}{\longrightarrow}\bigoplus_{n\geq 0}S_{n}\ra \colim_{m\geq 0}S_{n}\ra 0
\]
is an exact sequence; this provides a distinguished triangle in the associated derived category, and also in $\DM(k,R)$, which exhibits the colimit as a cone of the map $\id-\sigma$, and thus as a homotopy colimit.
\end{proof}

Nevertheless homotopy colimits are functorial in the following relatively weak sense.

\begin{lemma}\label{lemma funct hocolim}
For two $\NN$-indexed inductive (resp.\ projective) systems $F_*, \widetilde{F}_*$ (resp.\ $G^*, \widetilde{G}^*$) in $\DM(k,R)$, fix a choice of homotopy (co)limits $F, \widetilde{F}$ (resp.\ $G, \widetilde{G}$); that is, a specific choice of cones of the morphisms in Definition \ref{defn_hocolim}. Then, modulo these choices, there is a uniquely determined short exact sequence
\[
0 \ra R^1 \lim_n \Hom(F_n[1], \widetilde{F}) \ra \Hom( F, \widetilde{F})\ra \lim_n \Hom(F_n,\widetilde{F}) \ra 0
\]
(resp.\ $0 \ra R^1\lim_n \Hom( G,\widetilde{G}^n[-1])  \ra \Hom(G, \widetilde{G}) \ra \lim_n \Hom( G,\widetilde{G}^n) \ra 0$). In particular, for a morphism $f_*:F_*\ra \widetilde{F}_*$ of inductive systems, we can choose a morphism $f: F \ra \widetilde{F}$ such that for any $n\in \NN$ the diagram
\[
  \xymatrix{
    F_n \ar[r]^{f_n} \ar[d] & \widetilde{F}_n \ar[d] \\
    F \ar[r]^{f} &  \widetilde{F} 
    }
  \]
commutes, and $f$ is uniquely determined up to the $R^1\lim$ term appearing in the above exact sequence (and there is a similar statement for homotopy limits). By abuse of notation, we sometimes denote such a morphism by $M(f_*)$.
\end{lemma}
\begin{proof}
This follows directly from the definition of the functor $R^1\lim$.
\end{proof}  

The compatibility between the weak functoriality and the triangulated structure is as follows.

\begin{lemma}\label{lemma hocolim triangles}
Consider an $\NN$-indexed system of distinguished triangles $F'_*\ra F_*\ra F''_*\stackrel{+}{\rightarrow}$ in $\DM(k,R)$. For any choice of homotopy colimits of those systems and compatible morphisms between them as in Lemma \ref{lemma funct hocolim}, the triangle
  \[
    \hocolim F'_n \ra \hocolim F_n\ra \hocolim F''_n\stackrel{+}{\ra}
  \]
is distinguished. The analogous statement holds for homotopy limits.
\end{lemma}
\begin{proof}
This follows directly from the fact that direct sums (resp.\ direct products) of distinguished triangles are distinguished and the nine lemma.
\end{proof}  

Let us state some results about simple homotopy (co)limits.

\begin{lemma}\label{pulling out constants hocolims}  
For an inductive system $F_*: \NN\rightarrow \DM(k,R)$ and $A \in \DM(k,R)$, there is an isomorphism
\[ \hocolim_n  (F_n \otimes A) \simeq (\hocolim_n  F_n) \otimes A.  \]
\end{lemma}
\begin{proof} 
This follows from the definition of the homotopy colimit, as the tensor product commutes with direct sums in a tensor triangulated category. 
\end{proof}

\begin{lemma}\label{lemma hocolims sums}  
Let $I$ be a set and let  $F^i_*: \NN\rightarrow \DM(k,R)$ be an inductive system for all $i\in I$. Then there is an isomorphism
\[ \hocolim_n \bigoplus_{i\in I} F^i_n \simeq \bigoplus_{i\in I} \hocolim_n  F^i_n.\]
\end{lemma}
\begin{proof} 
This follows from the definition of the homotopy colimit, as a direct sum of distinguished triangles is distinguished. 
\end{proof}

\subsection{Vanishing results for homotopy (co)limits}
\label{sec:vanish-results-homot}

In order to compute homotopy (co)limits, we will frequently rely on various vanishing results which we collect together in this section. 

\begin{defn}\label{defn dim filtr}
The dimensional filtration on $\DM(k,R)$ is the $\ZZ$-indexed filtration $\ldots\subset\DM(k,R)_m\subset \DM(k,R)_{m+1}\subset\ldots$ where $\DM(k,R)_m$ denotes the smallest localising subcategory of $\DM(k, R)$ 
containing $M^c(X)(n)$ for all separated schemes $X$ of finite type over $k$ and all integers $n$ with $\dim(X) + n \leq m$. 
\end{defn}

Note that analoguous filtrations appear in the literature dealing with classes of stacks in the Grothendieck ring of varieties (for example, see \cite{BD}). Totaro proves the following result.

\begin{prop}[{\textsc{\cite[Lemma 8.3]{totaro}}}]
\label{prop totaro dim filtr}
The dimension filtration satisfies
\[ \bigcap_{m \leq 0} \DM(k,R)_m \simeq 0.\]
Moreover, for a projective system $G^*:\NN^{\text{op}}\rightarrow \DM(k,R)$ and a sequence of integers $a_n \ra - \infty$ such that $G^n \in \DM(k,R)_{a_n}$ for all $n$, it follows that $\holim_n G^n \simeq 0$.
\end{prop}

\begin{cor}\label{cor sum prod}
  Let $(M_n)_{n\in\NN}\in \DM(k,R)^{\NN}$ be a family of motives and $(a_n)_{n\in \NN}$ be a sequence of integers such that $a_n \ra - \infty$ and $M_n\in \DM(k,R)_{a_n}$ for all $n$. Then the natural morphism
  \[
\bigoplus_{n\in \NN} M_n \ra \prod_{n\in\NN} M_n
\]
is an isomorphism.
\end{cor}
\begin{proof}
The cone of this morphism lies in $\DM(k,R)_{a_n}$ for all $n\in \NN$, thus is zero by Proposition \ref{prop totaro dim filtr} (for further details, see the proof of \cite[Lemma 8.5]{totaro}).
\end{proof}

We will also need a vanishing result for homotopy colimits. However, the dual result to Propositon \ref{prop totaro dim filtr} does not hold in $\DM(k,R)$: if $k$ has infinite transcendence degree, there is an inductive system ( $\dots R(n)[n] \ra R(n+1)[n+1] \dots$ ) whose homotopy colimit is non-zero (\textit{cf.}\ \cite[Lemma 2.4]{Ayoub_conjectures}). Hence, the intersection $\cap_{n\geq 0}\DM^{\eff}(k,R)(n)$ is non-zero. Fortunately, with some control over the Tate twists and shifts, we can prove the following vanishing result.

\begin{prop}\label{prop hocolim vanishes}
Let $U_* \subset X_*$ be an inductive system of open immersions of smooth finite type $k$-schemes; that is, we have inductive systems $U_*, X_*: \NN\rightarrow \Sm_k$ and a morphism $U_*\ra X_*$ such that $U_n \hookrightarrow X_n$ is an open immersion for all $n \in \NN$. Let $c_n$ be the codimension of the complement $X_n -U_n$ in $X_n$. If $c_n \ra \infty$, then the morphism
\[ \colim_{n}R_{\tr}(U_n) \ra  \colim_{n}R_{\tr}(X_n)\]
is an $\AA^1$-weak equivalence. 
\end{prop}
\begin{proof}
By \cite[Proposition 8.1.(d)]{Cisinski_Deglise_cdh} and our standing assumptions on $k$ and $R$, we can assume the field $k$ to be perfect. Let $C_n$ be the level-wise mapping cone (as a $T$-spectrum of complexes of sheaves with transfers) of the morphism $R_{\tr}(U_n)\ra R_{\tr}(X_n)$. There are induced maps $C_n\ra C_{n+1}$ and we get a distinguished triangle
\[
\colim_{n} R_{\tr}(U_{n})\ra \colim_{n} R_{\tr}(X_{n})\ra\colim_{n} C_{n}\stackrel{+}{\ra}
\]
in $\DM(k,R)$. Thus it suffices to show that $\colim_{n} C_n$ is $\AA^1$-weakly equivalent to $0$.

On $W_n:=X_n- U_n$, consider an increasing open filtration $W_n^0 \subset W_n^{1} \subset \cdots \subset W_n^{m_n} = W_n$ with smooth differences $\partial W_n^j= W_n^j  - W_n^{j-1}$  (which exists by applying \cite[Tag 056V]{stacks_project} inductively, using that $k$ is perfect). For each $1\leq j\leq m_n$, we have a Gysin distinguished triangle
\[
M(U_n \cup W_n^{j-1})\ra M(U_n \cup W_n^{j})\ra M(\partial W^{j}_n)\{\codim_{X}(\partial W^j_n)\}.
\]
By inductively applying the octahedral axiom to these distinguished triangles, we conclude that $C_n$ is a successive extension of the motives $M(\partial W^{j}_n)\{\codim_{X}(\partial W^j_n)\}$.

Since the category $\DM(k,R)$ is compactly generated by motives of the form $M(X)(a)$, with $X \in \Sm_k$ and $a\in\ZZ$, it suffices to show for all $X \in \Sm_k$ and $a,b \in \ZZ$ that 
\[
\Hom(M(X)(a)[b],\colim_n C_n)= 0.
\]
As $M(X)(a)[b]$ is compact and filtered colimits are homotopy colimits in $\DM(k,R)$, we have
\[
\Hom(M(X)(a)[b],\colim_n C_n)=\colim_n \Hom(M(X)(a)[b],C_n).
\]
Since each cone $C_n$ is a successive extension of motives of the form $M(Z)\{c\}$ for $Z \in \Sm_k$ and $c \geq c_n$, it suffices to show for $n >\!> 0$ that for all $Z \in \Sm_k$ and $c \geq c_n$
\[
\Hom(M(X)(a)[b],M(Z)\{c\})=0.
\]
Equivalently, by Poincar\'e duality, it suffices to show
\begin{equation}\label{button}
\Hom(M(X)\otimes M^c(Z),R(c+\dim(Z)-a)[2c+2\dim(Z)-b])=0.
\end{equation}
As $c_n \ra \infty$, we have $c_n - b+a -\dim(X)>0$ for $n >\!> 0$. Thus, for $n >\!> 0$ and $c \geq c_n$, 
\[
(2c+2\dim(Z)-b)-((c+\dim(Z)-a)+\dim(Z)+\dim(X))=c-b + a-\dim(X) 
\]
is strictly positive. Then we deduce that \eqref{button} holds by using Lemma~\ref{lemma vanishing mot coh} below.
\end{proof}

\begin{lemma}\label{lemma vanishing mot coh}
Let $X$ (resp.\ $Z$) be a variety of dimension at most $d$ (resp.\ $e$). For $l,m \in \ZZ$ with $l>m+d+e$, we have
\[
\Hom(M(X)\otimes M^c(Z),R(m)[l])=0.
\]
\end{lemma}
\begin{proof}
By our standing assumption on $k$ and $R$ and \cite[Proposition 8.1.(d)]{Cisinski_Deglise_cdh}, we can assume that $k$ is a perfect field. Let us first prove the claim when $Z$ is proper. Then $M^c(Z)\otimes M(X)\simeq M(Z\times_k X)$, as $M^c(Z)\simeq M(Z)$. We have
\[
\Hom(M(Z\times_k X),R(m)[l])=H^l_{\cdh}(Z\times_k X,R(m)_{\cdh}).
\]
where $R(k)_{\cdh}$ is the cdh-sheaffification of the Suslin-Voevodsky motivic complex \cite{kelly}. The cohomological dimension of $(Z\times_k X)_{\cdh}$ is at most $d+e$ by \cite[Theorem 5.13]{Suslin_Voevodsky_BK}, and the complex $R(k)_{\cdh}$ is $0$ in cohomological degrees greater than $m$, which implies the result.

We now turn to the general case. Let $\bar{Z}$ be any compactification of $Z$ (not necessarily smooth). From the localisation triangle for the closed pair $(\bar{Z}- Z, \bar{Z})$, we obtain a long exact sequence
\begin{align*}
\ra \Hom(M(X)\otimes M^c(\bar{Z}),R(m)[l])\ra & \Hom(M(X)\otimes M^c(Z),R(m)[l]) \\
&\ra  \Hom(M(X)\otimes M^c(\bar{Z}- Z),R(m)[l-1]) \ra
\end{align*}
As both $\bar{Z}$ and $\bar{Z}-Z$ are proper with dimensions bounded by $e$ and $e-1$ respectively, we deduce that the outer terms vanish by the proper case, and so we obtain the desired result.
\end{proof}

\subsection{Motives of stacks}
\label{sec motive stacks}
The definition of motives of stacks in general is complicated by the fact that the category $\DM(k,R)$ does not satisfy descent for the \'etale topology (as this is already not the case for Chow groups); hence naive approaches to defining motives of even Deligne-Mumford stacks in terms of an atlas do not work. In Appendix~\ref{app mot stacks}, we explain and compare alternative approaches for defining \'etale motives of stacks.

To define the motive of a quotient stack $\fX = [X/G]$ independently of the presentation of $\fX$ as a quotient stack, we need an appropriate notion of ``algebraic approximation of the Borel construction $X\times^G EG$''. We use a variant of the definition of compactly supported motives of quotient stacks given by \cite{totaro} and extend this to more general stacks. More precisely, we will define the motive of certain smooth stacks over $k$, possibly not of finite type, which are exhaustive in the following sense.

\begin{defn}\label{def exh seq}
Let $\fX_0\stackrel{i_0}{\ra} \fX_1\stackrel{i_1}{\ra}\ldots\subset$ be a filtration of an algebraic stack $\fX$ by increasing open substacks $\fX_i \subset \fX$ which are quasi-compact and cover $\fX$; we will simply refer to this as a filtration. Then an exhaustive sequence of vector bundles on $\fX$ with respect to this filtration is a pair $(V_\bullet, W_\bullet)$ given by a sequence of vector bundles $V_m$ over $\fX_m$ together with injective maps of vector bundles $f_m:V_m\ra V_{m+1}\times_{\fX_{m+1}} \fX_m$ and closed substacks $W_m\subset V_m$ such that
\begin{enumerate}[label={\upshape(\roman*)}]
\item \label{lulo} the codimension of $W_m$  in $V_m$ tend towards infinity,
\item \label{guanabana} the complement $U_m:= V_m - W_m$ is a separated finite type $k$-scheme, and
\item \label{papaya}we have $f_m^{-1}(W_{m+1}\times_{\fX_{m+1}} \fX_m)\subset W_m$ (so that $f_m(U_m)\subset U_{m+1}\times_{\fX_{m+1}} \fX_m$).
\end{enumerate}
A stack admitting an exhaustive sequence with respect to some filtration is said to be exhaustive.
\end{defn}

For this paper, we have in mind two important examples of exhaustive stacks: i) quotient stacks (\textit{cf.}\ Lemma \ref{lemma quot exhaust}), and ii) the stack of vector bundles on a curve (\textit{cf.}\ Proposition \ref{prop exh seq}).

\begin{lemma}\label{lemma quot exhaust}
Let $\fX=[X/G]$ be a quotient stack of a quasi-projective scheme $X$ by an affine algebraic group $G$ such that $X$ admits a $G$-equivariant ample line bundle; then $\fX$ is exhaustive.
\end{lemma}
\begin{proof}
Let $\fX=[X/G]$ be a quotient stack; then there is an exhaustive sequences of vector bundles over $\fX$ with respect to the constant filtration built from a faithful $G$-representation $G\ra\GL(V)$ such that $G$ acts freely on an open subset $U\subset V$. More precisely, we let $W:=V -  U\subset V$ and for $m \geq 1$ consider the $G$-action  on $V^m$, which is free on the complement of $W^m$.  First, we construct an exhaustive sequence $(B_\bullet,C_\bullet)$ of vector bundles over $BG$ (with respect to the constant filtration) by taking $B_m=V^m$ with the natural transition maps and $C_m=W^m$. We then form the exhaustive sequence 
\[([(X \times B_\bullet)/G],[(X \times C_\bullet)/G])\] 
on $[X/G]$; this works as the open complement $(X \times (B_m - C_m)) /G$ is a quasi-projective scheme by \cite[Proposition 7.1]{mumford} due to the existence of a $G$-equivariant ample line bundle on $X$.
\end{proof}

\begin{defn}\label{def mot stack}
Let $\fX$ be a smooth exhaustive stack; then for an exhaustive sequence $(V_\bullet,W_\bullet)$ of vector bundles with respect to a filtration $\fX=\cup_m \fX_m$, we define the motive of $\fX$ as 
\[ M(\fX) := \colim_{m} R_{\tr}(U_m) \]
with transition maps induced by the composition of $f_m : U_m \ra U_{m+1}\times_{\fX_{m+1}}\fX_m$ (by Definition \ref{def exh seq} \ref{papaya} $f_m$ restricts to such a morphism) with the projection $U_{m+1}\times_{\fX_{m+1}}\fX_m\ra U_{m+1}$, and the colimit is taken in the category of $T$-spectra of complexes of sheaves with transfers.
\end{defn}

By Lemma~\ref{colim_hocolim}, this definition implies the following lemma.

\begin{lemma}\label{stack colim hocolim}
In $\DM(k,R)$, for $\fX$ and  $(V_\bullet,W_\bullet)$ as in Definition \ref{def mot stack}, we have an isomorphism
\[
M(\fX)\simeq \hocolim M(U_m).
\]
\end{lemma}

The smoothness of $\fX$ is required to prove this definition is independent of the above choices. We prove this in two steps.

\begin{lemma}\label{lem exhaust filtr}
 Let $\fX$ be an exhaustive stack and $(\fX'_{n})_{n\in \NN}$ be some increasing filtration by quasi-compact open substacks. Then there exists a subsequence $(\fX'_{n_{i}})_{i\in \NN}$ of $(\fX'_{n})$ and an exhaustive sequence of vector bundles with respect to this filtration. Moreover, this construction can be made compatibly with any fixed exhaustive sequence of vector bundles with respect to a filtration $\fX_{n}$, in the sense that there is a subsequence $\fX_{n_{i}}$ such that
  \[
\fX'_{n'_1}\subset \fX_{n_1}\subset \fX'_{n'_2}\subset \fX_{n_2} \ldots
\]
and an exhaustive sequence of vector bundles for that new filtration which restricts to the two existing exaustive sequence.
\end{lemma}  

\begin{proof}
Since $\fX$ is exhaustive, we can find a filtration $(\fX_n)_{n\geq 0}$ by quasi-compact open substacks  and an exhaustive sequence of vector bundles $(V_{n},W_{n})_{n\geq 0}$ with respect to this filtration. Note that by quasi-compactness of the $\fX'_n$ and the fact that $\fX= \cup_n \fX_n$, for each $n'\in \NN$, there exists $n\in \NN$ such that $\fX'_{n'}\subset \fX_{n}$. By iterating this quasi-compactness argument, we find strictly increasing sequences $(n_i)_{i \geq 0}$, $(n'_i)_{i \geq 0}$ such that $\fX'_{n'_1}\subset \fX_{n_1}\subset \fX'_{n'_2}\subset \fX_{n_2}\ldots \subset \fX$. Let us construct an exhaustive sequence of vector bundles adapted to $\fX'_{n'_1}\subset \fX_{n_1}\subset \fX'_{n'_2}\subset \fX_{n_2} \ldots \subset \fX$ On the $\fX_{n_{i}}$'s we take $(V_{n_{i}},W_{n_{i}})$ while on the $\fX'_{n'_{i}}$'s we take $(V_{n_{i}}\times_{\fX_{n_{i}}}\fX'_{n'_{i}},W_{n_{i}}\times_{\fX_{n_{i}}}\fX'_{n'_{i}})$, with the obvious transition maps. It is easy to see that this satisfies the conditions of Definition \ref{def exh seq}. Finally, to get an exhaustive sequence with respect to the subsequence filtration $\fX'_{n'_{i}}$, one takes the same pairs $(V_{n_{i}}\times_{\fX_{n_{i}}}\fX'_{n'_{i}},W_{n_{i}}\times_{\fX_{n_{i}}}\fX'_{n'_{i}})$ and one compose the transition maps of the previous sequence two by two. This concludes the proof.
\end{proof}

\begin{lemma}\label{lem stack def indept}
The motive of a smooth exhaustive stack $\fX$ does not depend (up to a canonical isomorphism in $\DM(k,R)$) on the choice of the filtration or the exhaustive sequence of vector bundles.
\end{lemma}
\begin{proof}
We first fix a filtration of $\fX=\cup_m \fX_m$ and prove the resulting object does not depend on the exhaustive sequence. Let $(V_\bullet,W_\bullet)$ and $(V'_\bullet,W'_\bullet)$ be two exhaustive sequence of vector bundles on $\fX$. As in the proof of \cite[Theorem 8.4]{totaro}, we introduce a new sequence $(V''_\bullet,W''_\bullet)$ with
\begin{equation}\label{new_exhaust}
(V''_m,W''_m):= (V_m\times_{\fX_m} V'_m,W_m\times_{\fX_m} W'_m).
\end{equation}
and transition morphisms $f''_m:=f_m\times f'_m:V_m\times_{\fX_m} V'_m\ra (V_{m+1}\times_{\fX_{m+1}}\fX_m )\times_{\fX_m} (V'_{m+1}\times_{\fX_{m+1}}\fX_m )$, which are injective vector bundle homomorphisms. Let us prove that this sequence satisfies properties \ref{guanabana} and \ref{papaya} in Definition \ref{def exh seq} (in fact, it satisfies property \ref{lulo}, but this is not needed below). For \ref{guanabana}, we note that $U''_m:=V''_m - W''_m = U_m\times_{\fX_m} V'_m \cup V_m\times_{\fX_m} U'_m$ is a separated scheme of finite type, as $U_m\times_{\fX_m} V'_m$ (resp.\ $V_m\times_{\fX_m} U'_m$) is a vector bundle over the scheme $U_m$ (resp.\ $U'_m$). For \ref{papaya}, we have
\[
(f''_m)^{-1}(W''_{m+1}\times_{\fX_{m+1}}\fX_m )= (f_m^{-1}(W_{m+1}\times_{\fX_{m+1}}\fX_m)) \times_{\fX_m} ((f'_m)^{-1}(W'_{m+1}\times_{\fX_{m+1}}\fX_m))\subset W''_m.
  \]
Given these two properties, one can define a system of $T$-spectra $\dots \ra R_{\tr}(U''_m) \ra R_{\tr}(U''_{m+1})\ra \dots$ as in Definition \ref{def mot stack} and to complete the proof, it suffices by symmetry to show that $\colim_{m} R_{\tr}(U''_{m})\simeq \colim_{m} R_{\tr}(U_{m})$ in $\DM(k,R)$.

We note that $U_m\times_{\fX_m} V'_m$ is open in $U''_m$ (thus a smooth scheme, as $\fX$ is smooth) and that the codimension of the complement satisfies
\[
\codim_{U''_m} (W_m\times_{\fX_m} U'_m) = \codim_{V''_m} (W_m\times_{\fX_m} U'_m) \geq \codim_{V''_m} (W_m\times_{\fX_m} V'_m) = \codim_{V_m} (W_m)  
\]
where we have used that $U''_m$ is dense open in $V''_m$ and that $V'_m\ra \fX_m$ is a vector bundle. In particular, these codimensions tend to infinity with $m$, and so we can apply  Proposition~\ref{prop hocolim vanishes} to the inductive system of open immersions $U_m\times_{\fX_m} V'_m \hookrightarrow U''_m$ of smooth schemes. Since $U_m\times_{\fX_m} V'_m\ra U_m$ is a vector bundle, we have two morphisms of inductive systems of $T$-spectra
  \[
\xymatrix@R=1.1em{
  \ldots \ar[r] &  R_{\tr}(U''_m) \ar[r]  & R_{\tr}(U''_{m+1}) \ar[r]  & \ldots\\
  \ldots \ar[r] &  R_{\tr}(U_m\times_{\fX_m} V'_m) \ar[r] \ar[d]_{\AA^1-\mathrm{w.e.}} \ar[u] & R_{\tr}(U_{m+1}\times_{\fX_{m+1}} V'_{m+1}) \ar[r] \ar[d]_{\AA^1-\mathrm{w.e.}} \ar[u] & \ldots\\
  \ldots \ar[r] &  R_{\tr}(U_m) \ar[r] & R_{\tr}(U_{m+1}) \ar[r] & \ldots
}
\]
whose colimits are both $\AA^1$-weak equivalences (by Proposition~\ref{prop hocolim vanishes} and the fact that a filtered colimit of $\AA^1$-weak equivalences is an $\AA^1$-weak equivalence). 
    
We now show that the definition is independent of the filtration by quasi-compact open substacks. Let $(\fX_n)_{n\geq 0}$ and $(\fX'_n)_{n\geq 0}$ be two such filtrations of $\fX$. By Lemma \ref{lem exhaust filtr}, we can find subsequences with
\[
\fX'_{n'_1}\subset \fX_{n_1}\subset \fX'_{n'_2}\subset \fX_{n_2} \ldots
\]
and an exhaustive sequence of vector bundles with respect to that new filtration. Applying several times the fact that $\NN$-indexed homotopy colimits are unchanged by passing to a subsequence finishes the proof.
\end{proof}

If $\fX=X$ is a separated scheme, then $X$ is exhaustive (as we can take $\fX_m =X = V_m = U_m$) and this definition coincides with the usual definition of the motive of $X$.

\begin{ex}\label{ex mot BGm}
The compactly supported motive of the classifying space $B\GG_m$ of the multiplicative group $\GG_m$ is computed by Totaro \cite[Lemma 8.5]{totaro}, by using the family of representations $\rho_n : \GG_m \ra \AA^{n}$ given by $t \mapsto \text{diag}(t, \dots ,t)$. More precisely, we have an exhaustive sequence over $B\GG_m$ given by $(V_n:=[\AA^n/\GG_m],W_n := [ \{ 0 \}/\GG_m])_{n\in \NN}$, and we can also use this to compute the motive of $B\GG_m$. The open complement is $U_n := [(\AA^n - \{ 0 \} )/\GG_m] \simeq \PP^{n-1}$ and so we have
\[ M(B\GG_m) = \colim_n R_{\tr}(\PP^{n-1})\simeq \hocolim_{n}M(\PP^{n-1}) \simeq \bigoplus_{j \geq 0 } R\{ j \}.\]
\end{ex}

\begin{prop}\label{prop mor compact}
  Let $\fX$ be a smooth exhaustive stack. Fix an exhaustive sequence $(V_\bullet,W_\bullet)$ of vector bundles with respect to a filtration $\fX=\cup_m \fX_m$ so that $M(\fX)=\colim_m R_{\tr}(U_m)$. Let $N\in \DM_c(k,R)$ be any compact motive. Then we have isomorphisms
\begin{enumerate}[label={\upshape(\roman*)}]
\item $\Hom(N,M(\fX))\simeq \colim_n \Hom(N,M(U_n))$ and  
\item $\Hom(M(\fX), N)\simeq \lim_n \Hom(M(U_n),N)$.
\end{enumerate}
\end{prop}

\begin{proof}
Part (i) holds, as filtered colimits in $\DM(k,R)$ are homotopy colimits and $N$ is compact. For (ii), by Lemma \ref{lemma funct hocolim}, it suffices to show that $R^1\lim \Hom(M(U_n)[1],N) =0$. We actually show that the map $\Hom(M(U_{n+1})[1],N)\ra \Hom(M(U_{n})[1],N)$ is an isomorphism for $n$ large enough, which implies that the corresponding $R^{1}\lim$ term vanishes.

Since $N$ is compact, there exists $m\in\ZZ$ such that $N\in\DM(k,R)_m$ (see Definition \ref{defn dim filtr}). By the argument in the proof of Proposition \ref{prop hocolim vanishes}, for $n$ large enough, the cone of $M(U_n)\ra M(V_n)$ is in the triangulated subcategory generated by motives of the form $M(Z)\{c\}$ with $Z\in\Sm_{k}$ and $c>m$. Then by \cite[Lemma 8.1]{totaro}, this implies that $\Hom(\text{Cone}(M(U_n)\ra M(V_n)),N[i])=0$ for all $i\in \ZZ$, thus in particular $\Hom(M(U_n)[1],N)\simeq \Hom(M(V_n)[1],N)$.

By definition, the transition maps in the system $\Hom(M(V_n)[1],N)$ are induced by the maps $M(V_n)\ra M(i_n^* V_{n+1})\ra M(V_{n+1})$ where $i_{n}$ is the open immersion $\fX_{n}\hookrightarrow \fX_{n+1}$. Both $V_n$ and $i_n^* V_{n+1}$ are vector bundles over the same stack $\fX_n$, and $V_n \ra i_n^*V_{n+1}$ is a map of vector bundles, so $M(V_n)\ra M(i_n^*V_{n+1})$ is an isomorphism. On the other hand, $i_n^* V_{n+1}\ra V_{n+1}$ is an open immersion, and $\codim_{V_{n+1}}(V_{n+1} - i_n^* V_{n+1})= \codim_{\fX_{n+1}}(\fX_{n+1} - \fX_n)$ which tends to infinity with $n$ by assumption. To complete the proof, we use the same argument as in the paragraph above to deduce that $\Hom(M(V_{n+1})[1],N)\simeq \Hom(M(V_n)[1],N)$ for $n$ large enough.
\end{proof}

This shows, in particular, that the following definition is not unreasonable (since it can be computed via any given exhaustive sequence).

\begin{defn}\label{def mot coh}
Let $\fX$ be a smooth exhaustive stack, and $p,q\in \ZZ$. The motivic cohomology of $\fX$ is defined as
$H^p(\fX,R(q)):= \Hom (M(\fX),R(q)[p])$. \end{defn}  

It is not immediately clear that the definition of the motive of an exhaustive stack is functorial, but it is relatively simple to prove a weak form of functoriality for certain representable morphisms. To have more systematic forms of functoriality it is better to work with other definitions of motives of stacks as explained in Appendix~\ref{app mot stacks}.

\begin{lemma}\label{funct rep quot stacks}
  Let $g: \fX \ra \cY$ be a flat finite type representable morphism (resp. an open immersion) of smooth algebraic stacks such that $\cY$ is exhaustive, with an exhaustive sequence of vector bundles $(V_\bullet,W_\bullet)$ on $\cY$ with respect to a filtration $\cY=\cup_m \cY_m$. Then $\fX$ is exhaustive, with the exhaustive sequence $(g^*V_\bullet,g^{-1}(W_\bullet))$ with respect to the filtration $\fX=\cup_m g^{-1}(\cY_m)$. Moreover, $g$ induces a morphism
  \[ M(g) : M(\fX) \ra M(\cY).\]
\end{lemma}
\begin{proof}
Since $g$ is flat, we have that $\text{codim}_{g^*V_n}(g^{-1}(W_n))\geq \text{codim}_{V_n}(W_n)$ which tends to infinity with $n$. By the fact that $g$ is representable and of finite type (resp. an open immersion),
\[g^*V_n- g^{-1}(W_n)\simeq (V_n- W_n)\times_\cY \fX\] is a separated finite type $k$-scheme. We leave the verification of Property \ref{papaya} in Definition~\ref{def exh seq} to the reader. The morphism $M(g)$ is then defined by taking colimits in the morphism of systems of $T$-spectra $R_{\tr}(g^{-1}(U_n))\ra R_{\tr}(U_n)$.
\end{proof}
 
\begin{rmk}
The flatness condition on $g$ was only imposed to prove the codimension condition on $(g^*V_\bullet,g^{-1}(W_\bullet))$. There are many other cases in which the definition of $M(f)$ above works; for instance, for any finite type representable morphism between quotient stacks.  
\end{rmk}

Let us record a first application of this weak functoriality.

\begin{lemma}\label{lem hocolim substacks}
  Let $\fX$ be a smooth exhaustive stack which admits an increasing filtration $\fX = \cup_{i \in \NN} \fX_i$ by quasi-compact open substacks $\fX_i$. Fix an exhaustive sequence of vector bundles $(V_{\bullet},W_{\bullet})$ with respect to an exhaustive filtration of $\fX$ (a priori different from $\fX_{i}$). For a given $i\in \NN$, this exhaustive sequence of vector bundles, pulled-back to $\fX_{i}$, provides an exhaustive sequence of vector bundles with respect to the trivial filtration on $\fX_{i}$. In particular, we have morphisms $M(\fX_{i})\to M(\fX_{i+1})$ induced by Lemma \ref{funct rep quot stacks}. For these morphisms, we have
\[
M(\fX)\simeq \hocolim_{i\in \NN} M(\fX_{i}).
\]
\end{lemma}  
\begin{proof}
  Everything except the last statement follows from Lemma \ref{funct rep quot stacks} applied to the open immersions $\fX_{i}\to \fX$ and $\fX_{i}\to \fX_{i+1}$. Write $U_{n}:=V_{n}\setminus W_{n}$ for $n\in\NN$, and $\widetilde{U}_{n,i}:=U_{n}\times_{\fX}\fX_{i}$ for $n,i\in\NN$. By quasi-compactness of the $\fX_{i}$ and the fact that they cover $\fX$, we have $\widetilde{U}_{n,i}=U_{n}$ for $i$ large enough (depending on $n$). Then
  \begin{align*}
    M(\fX) & \simeq \colim_{n\in\NN} R_{\tr}(U_{n})   \simeq \colim_{n\in\NN}\colim_{i\in \NN} R_{\tr}(\widetilde{U}_{n,i}) \\
    & \simeq \colim_{i\in\NN}\colim_{n\in\NN} R_{\tr}(\widetilde{U}_{n,i}) \simeq \hocolim_{i\in\NN} M(\fX_{i})
  \end{align*}
by definition and the fact that $\NN$-colimits of motivic spectra compute $\NN$-homotopy colimits.
\end{proof}  

One can also prove K\"{u}nneth isomorphisms and $\AA^1$-homotopy invariance, as well as a Gysin distinguished triangle.

\begin{prop}\label{kunneth}
Let $\fX$ and $\cY$ be smooth exhaustive stacks over $k$.
\begin{enumerate}[label={\upshape(\roman*)}]
\item The stack $\fX\times_{k} \cY$ is exhaustive and there is an isomorphism $M(\fX \times_k \cY) \simeq M(\fX) \otimes M(\cY)$.
\item If $\cE \ra \fX$ is a vector bundle, then $\cE$ is a smooth exhaustive stack and $M(\cE) \simeq M(\fX)$.
\item Let $i:\cZ\to \fX$ be a closed immersion of relative codimension $c$ with $\cZ$ smooth. Then there is a distinguished triangle
  \[
  M(\fX - \cZ)\rightarrow M(\fX)\stackrel{\Gy(i)}{\rightarrow} M(\cZ)\{c\}\stackrel{+}{\rightarrow}.
\]
\end{enumerate}
\end{prop}
\begin{proof}
For part (i), let $\fX = \cup_{n} \fX_n$ (resp.\ $\cY = \cup_{n} \cY_n$) be filtrations with respect to which $\fX$ (resp.\ $\cY$) admits an exhaustive sequence of vector bundles $(V_{\bullet},W_{\bullet})$ (resp.\ $(V'_{\bullet},W'_{\bullet})$). Then we leave the reader to easily verify that on $\fX\times_{k}\cY$ there is an exhaustive sequence compatible with the filtration $\fX\times_{k}\cY=\cup_{n}\fX_{n}\times \cY_{n}$ given by $(V_{\bullet}\times V'_{\bullet},W_{\bullet}\times_{k} V'_{\bullet} \cup V_{\bullet}\times_{k} W'_{\bullet})$. We thus get
\begin{align*}
  M(\fX\times_{k}\cY)  &= \colim_{n} R_{\tr}(U_{n}\times U'_{n}) \simeq \colim_{n} R_{\tr}(U_{n})\otimes R_{\tr}(U'_{n}) 
 \simeq \colim_{n} \colim_{m} R_{\tr}(U_{n})\otimes R_{\tr}(U'_{m}) \\ & \simeq \colim_{n} R_{\tr}(U_{n})\otimes \colim_{m} R_{\tr}(U'_{m})= M(\fX)\otimes M(\cY)
\end{align*}
where we have used the commutation of tensor products of $T$-spectra with colimits and the fact that the diagonal $\NN\ra \NN\times \NN$ is cofinal. The proof of part (ii) follows by a similar argument using the pullback of any exhaustive sequence from $\fX$ to $\cE$ and homotopy invariance for schemes.

Let us prove part (iii). Let us write $j:\cT\to \cX$ for the complementary open immersion to $i$. Pulling back the $\fX_{n}$ and $(V_{\bullet},W_{\bullet})$ along $i$ (resp. $j$) produces a filtration $\cZ_{n}$  (resp. $\cT_{n}$) and an exhaustive sequence of vector bundles $(V'_{\bullet},W'_{\bullet})$ for $\cZ$ (resp. $(V''_{\bullet},W''_{\bullet})$ for $\cT$). For every $n$, we get a closed immersion $i_{n}:U'_{n}\to U_{n}$ of smooth finite type $k$-schemes and a corresponding Gysin triangle
\[
  M(U''_{n})\rightarrow M(U_{n})\stackrel{\Gy(i_{n})}{\rightarrow} M(U'_{n})\{c\}\stackrel{+}{\rightarrow}.
\]
Moreover, for all $n\in \NN$, we have a cartesian square
\[
  \xymatrix{
    U'_{n} \ar[r] \ar[d] & U_{n} \ar[d]\\
    U'_{n+1} \ar[r] & U_{n+1}
  }
\]
of smooth $k$-schemes, so that by \cite[Proposition 4.3, Proposition 4.10]{Deglise_Gysin_II}, the Gysin triangle above is functorial in $n$, with transition morphisms induced by the morphisms $U''_{n}\to U''_{n+1}$, $U_{n}\to U_{n+1}$ and $U'_{n}\to U'_{n+1}$ respectively. We conclude the proof by passing to the homotopy colimit.
\end{proof}

\begin{rmk}\label{rmk motive diag}
In fact, the argument used in Proposition \ref{kunneth} (i) enables us to define a morphism 
\[ M(\Delta) : M(\fX) \ra M(\fX) \otimes M(\fX) \]
for any smooth exhaustive stack $\fX$. Indeed this morphism is defined as the colimit over $n$ of the morphisms 
$R_{\tr}(U_n) \ra R_{\tr}(U_n \times U_n) \simeq M(U_n) \otimes M(U_n)$, and one can check that this morphism is independent of the presentation.
\end{rmk}

\subsection{Chern classes of vector bundles on stacks}

For a vector bundle $E$ over a smooth $k$-scheme $X$, one has motivic incarnations of Chern classes given by a morphism
\[ c_j(E) : M(X) \ra R\{j \}.\]
Let us extend this notion to vector bundles on smooth exhaustive stacks.

\begin{defn}
Let $\cE \ra \fX$ be a vector bundle on a smooth exhaustive stack and let $(V_{\bullet},W_\bullet)$ be an exhaustive system of vector bundles on $\fX$ with respect to a filtration $\fX = \bigcup_n \fX_{n}$; then the pullback $\cE_n$ of $\cE$ to the smooth scheme $U_n:=V_n - W_n \hookrightarrow V_n \twoheadrightarrow \fX_n \hookrightarrow \fX$ determines a morphism $c_j(\cE_n) : M(U_n)  \ra R\{j \}$. Since these morphisms are compatible, this determines an element of $\lim_n \Hom( M(U_n) , R\{j\})$, which corresponds to a morphism
\[ c_j(\cE) : M(\fX) \ra R\{ j \} \]
by Proposition \ref{prop mor compact} as $R\{ j \}$ is compact. 
\end{defn}

\begin{lemma} 
This definition does not depend on any of the above choices.
\end{lemma}
\begin{proof}
We omit the details as the proof is very similar to that of Lemma \ref{lem stack def indept}.
\end{proof}

We will need some basic functoriality results for these Chern classes. 

\begin{lemma}
 Let $g: \fX \ra \cY$ be a flat finite type representable morphism of smooth algebraic stacks such that $\cY$ is exhaustive (then $\fX$ is exhaustive by Lemma \ref{funct rep quot stacks}). For a vector bundle $\cE \ra \cY$ and $j \in \NN$, we have a commutative diagram
 \[\xymatrix{
M(\fX) \ar[d]_{c_j(f^*\cE)} \ar[r]^{M(f)} & M(\cY)  \ar[d]^{c_j(\cE)} \\  R\{j\} \ar[r]^{=} & R\{j\}.
  }\] 
\end{lemma}
\begin{proof}
This follows from the proof of Lemma \ref{funct rep quot stacks} and the functoriality of Chern classes for vector bundles on smooth schemes. 
\end{proof}

\subsection{Motives of $\GG_m$-torsors}

In this section, we prove a result for the motive of a $\GG_m$-torsor, which will be used to compute the motive of the stack of principal $\SL_n$-bundles. We recall that a $\GG_m$-torsor over a stack $\fX$ is a morphism $\fX \ra B\GG_m$, or equivalently a cartesian square
\[\xymatrix{ \cY \ar[r] \ar[d] &  \spec k \ar[d] \\  \fX \ar[r] & B\GG_m.  }\] 

\begin{prop}\label{motives of mult bundles}
Let $\fX$ and $\cY$ be smooth exhaustive stacks and suppose that $\cY \ra\fX$ is a $\GG_m$-torsor. Let $\cL := [\cY \times \AA^1 /\GG_m]$ be the associated line bundle over $\fX$. Then there is a distinguished triangle
\begin{equation}\label{rabbit}
 M(\cY) \ra M(\fX) \ra M(\fX)\{ 1 \} \stackrel{+}{\rightarrow}
\end{equation}
where the morphism $M(\fX) \ra M(\fX)\{ 1 \}$ is the following composition
\[ \varphi_L: \xymatrix@1{ M(\fX)\ar[rr]^{M(\Delta)  \quad \quad} & &  M(\fX) \otimes M(\fX) \ar[rr]^{\quad \id \otimes c_1(\cL)} & &  M(\fX)\{ 1 \}. }\]
\end{prop}
\begin{proof}
Let us first verify the statement for schemes. Let $X \in \Sm_k$ and $Y \ra X$ be a $\GG_m$-torsor. If we let $L = Y \times^{\GG_m} \AA^1$ denote the associated line bundle; then $Y = L - X$. By $\AA^{1}$-homotopy invariance, we have $M(L)\simeq M(X)$. We can consider the Gysin triangle associated to the closed immersion $X \hookrightarrow L$ as the zero section:
\[ M(Y) \ra M(L) \simeq M(X) \ra M(X)\{ 1\} \stackrel{+}{\rightarrow}. \]
Then the map $M(X) \ra M(X)\{ 1\}$ is given by the first Chern class of $L$ by \cite[Example 1.25]{Deglise_Gysin_1}.

Now suppose we are working with smooth exhaustive stacks. We note that the morphism $M(\Delta)$ is defined in Remark \ref{rmk motive diag}. Let $(V_\bullet, W_\bullet)$ be an exhaustive sequence on $\fX$ with respect to a filtration $\fX = \cup_{n} \fX_n$ and let $U_n := V_n - W_n$ as usual. Then we let $L_n$ denote the pullback of the line bundle $\cL \ra \fX$ to $U_n$ and let $Y_n \subset L_n$ denote the complement to the zero section. Since $U_n$ are schemes, we have for all $n$, distinguished triangles
\begin{equation}\label{frog}
M(Y_n) \ra M(U_n) \ra M(U_n)\{ 1 \} \stackrel{+}{\rightarrow}.
\end{equation}
Since the fibre product $U_n \cong L_n \times_{L_{n+1}} U_{n+1}$ is transverse, the Gysin morphisms for $(U_n,L_n)$ and $(U_{n+1},L_{n+1})$ are compatible, and so by taking the homotopy colimit of the distinguished triangles \eqref{frog}, we obtain a distinguished triangle of the form \eqref{rabbit} and one can check that the morphism $M(\fX) \ra M(\fX)\{ 1 \}$ is $\varphi_L$ by definition of $M(\Delta)$ and the first Chern class. 
\end{proof}

\begin{ex}\label{ex univ Gm torsor}
Let us consider the universal $\GG_m$-torsor $\spec k \ra B \GG_m$; then the associated line bundle is $\cL = [\AA^1/\GG_m] \ra B\GG_m$. We recall that in Example \ref{ex mot BGm}, we used the exhaustive sequence $(V_n = [\AA^n/\GG_m], W_n=[\{0 \}/\GG_m])$ for the trivial filtration on $B\GG_m$ to show that
\[ M(B\GG_m) = \hocolim_n M(\PP^{n-1}) \simeq \bigoplus_{n \geq 0} R\{ n \}.\]
Using the notation of the proof above, the line bundle $L_n$ on $U_n=\PP^{n-1}$ is the tautological line bundle $\cO_{\PP^n}(-1)$, whose first Chern class $c_1(L_n) : M(\PP^{n-1})\simeq \oplus_{j=0}^{n-1} R\{ j \} \ra R\{1 \}$ is just the projection onto this direct factor. It follows from this that the morphism 
\[ \varphi_{\cL} : M(B \GG_m) \simeq \bigoplus_{n \geq 0} R\{ n \} \ra M(B \GG_m)\{1\} \simeq \bigoplus_{n \geq 1} R\{ n \}\]
is the natural projection.
\end{ex}

\subsection{Compactly supported motives and Poincar\'e duality}\label{sec mot comp}

One can define the compactly supported motive of an exhaustive stack as follows.

\begin{defn}\label{def compactly supported motive}
Let $\fX$ be an exhaustive algebraic stack. For an exhaustive sequence $(V_\bullet, W_\bullet)$ of vector bundles on $\fX$ with respect to a filtration $\fX = \cup_m \fX_m$, we define the compactly supported motive of $\fX$ by
\[ M^c(\fX):= \holim_m M(U_m)\{ - \rk(V_m)\} \]
with transition maps given by the composition 
\[ M^c(U_{m+1})\{ -\rk(V_{m+1})\} \ra M^c(U_{m+1}\times_{\fX_{m+1}}\fX_m )\{ -\rk(V_{m+1})\} \ra M^c(U_m) \{ - \rk(V_m)\}
 \]
where the first map is the flat pullback for the open immersion $U_{m+1}\times_{\fX_{m+1}}\fX_m  \ra U_m$ and the second map is defined using contravariant functoriality of $M^c$ for the regular immersion $f_m : U_m \hookrightarrow U_{m+1}\times_{\fX_{m+1}}\fX_m $ of relative dimension $\rk(V_{m+1}) - \rk(V_m)$ (\textit{cf.}\ \cite[$\S$4.3]{Deglise_Gysin_1}). 
\end{defn}

\begin{rmk}
One can show that this definition is independent of these choices similarly to Lemma \ref{lem stack def indept} (\textit{cf.}\ \cite[Theorem 8.4]{totaro}). For the compactly supported motive, we do not have to assume that $\fX$ is smooth (as the argument uses Proposition \ref{prop totaro dim filtr} instead of Proposition \ref{prop hocolim vanishes}).
\end{rmk}

We can show one part of the statement of Poincar\'{e} duality for exhaustive smooth stacks follows from Poincar\'{e} duality for schemes.

\begin{prop}[Poincar\'{e} duality for exhaustive stacks]\label{prop PD for stacks}
Let $\fX$ be a smooth exhaustive stack of dimension $d$; then in $\DM(k, R)$, there is an isomorphism
\[ M(\fX)^\vee \simeq M^c(\fX)\{-d\}.\]
\end{prop}
\begin{proof}
This follows directly from the definitions of $M(\fX)$ and $M^c(\fX)$ and Poincar\'e duality for schemes, as the dual of a homotopy colimit is a homotopy limit.
\end{proof}

Note that, because the dual of an infinite product is not in general an infinite sum, it is not clear in general that the duality works the other way.

\section{The motive of the stack of vector bundles}
\label{sec motive bun}
Throughout this section, we let $C$ be a smooth projective geometrically connected curve of genus $g$ over a field $k$. We fix $n \in \NN$ and $d \in \ZZ$ and let $\Bun_{n,d}$ denote the stack of vector bundles over $C$ of rank $n$ and degree $d$; this is a smooth stack of dimension $n^2(g-1)$.  In this section, we give a formula for the motive of $\Bun_{n,d}$ in $\DM(k,R)$, by adapting a method of Bifet, Ghione and Letizia \cite{bgl} to study the cohomology of $\Bun_{n,d}$ using matrix divisors. This argument was also used by Behrend and Dhillon \cite{BD} to give a formula for the virtual motivic class of $\Bun_{n,d}$ in (a completion of) the Grothendieck ring of varieties (see $\S$\ref{sec compare}). 

We can define the motive of $\Bun_{n,d}$, as it is exhaustive: to explain this, we filter $\Bun_{n,d}$ using the maximal slope of all vector subbundles. 

\begin{defn}
The slope of a vector bundle $E$ over $C$ is $\mu(E):= \frac{\deg(E)}{\rk(E)}$. We define
\[ \mu_{\max}(E) = \max \{ \mu(E') : 0 \neq E' \subset E \}. \]
\end{defn}

The maximal slope $\mu_{\max}(E)$ is equal to the slope of the first vector bundle appearing in the Harder--Narasimhan filtration of $E$. For $\mu \in \QQ$, we let $\Bun_{n,d}^{\leq \mu}$ denote the substack of $\Bun_{n,d}$ consisting of vector bundles $E$ with $\mu_{\max}(E) \leq \mu$; this substack is open by upper semi-continuity of the Harder--Narasimhan type \cite{shatz}. Any sequence $(\mu_l)_{l\in \NN}$ of increasing rational numbers tending to infinity defines a filtration of $(\Bun^{\leq \mu_l}_{n,d})_{l \in \NN}$ of $\Bun_{n,d}$. The stacks $\Bun_{n,d}^{\leq \mu_l}$ are all quasi-compact, as they are quotient stacks. Indeed, all vector bundles over $C$ of rank $n$ and degree $d$ with maximal slope less than or equal to $\mu_l$ form a bounded family (\textit{cf.}\ \cite[Theorem 3.3.7]{HL}), and so can be parametrised by an open subscheme $Q^{\leq \mu_l}$ of a Quot scheme, and then $\Bun_{n,d}^{\leq \mu_l} \simeq [Q^{\leq \mu_l}/\GL_N]$ (for further details, see for example \cite[Th\'{e}or\`{e}me 4.6.2.1]{laumon_MB}).

We will construct an exhaustive sequence of vector bundles on $\Bun_{n,d}$ by using matrix divisors.

\subsection{Matrix divisors}

A matrix divisor (after Weil) of rank $n$ and degree $d$ on $C$ is a locally free 
subsheaf of $\cK^{\oplus n}$ of rank $n$ and degree $d$, where $\cK$ denotes the 
constant $\cO_C$-module equal to the function field of $C$. 

Let $D$ be an effective divisor on $C$ and let $\Div_{n,d}(D)$ denote the scheme 
parametrising locally free subsheaves of $\cO_C(D)^{\oplus n}$ of rank $n$ and degree $d$. Equivalently, $\Div_{n,d}(D)$ is the Quot scheme
\[ \Div_{n,d}(D) := \quot_C^{n \deg D -d}(\cO_C(D)^{\oplus n}), \]
parametrising degree $n \deg D -d$ torsion quotient sheaves of $\cO_C(D)^{\oplus n}$. Thus $\Div_{n,d}(D)$ is a projective variety, and it is also smooth, as it parametrises torsion quotient sheaves on a curve (\textit{cf.}\ \cite[Proposition 2.2.8]{HL}); moreover, it has dimension $n^2 \deg D -nd$ by the Riemann-Roch formula. For effective divisors $D' \geq D \geq 0$ on $C$, there is a natural closed immersion
\[i_{D,D'}:\Div_{n,d}(D) \ra \Div_{n,d}(D') \]
compatible with the forgetful morphisms to $\Bun_{n,d}$, and so we can construct an ind-variety $\Div_{n,d}:= (\Div_{n,d}(D))_{D}$ of matrix divisors of 
rank $n$ and degree $d$ on $C$.

We will use the forgetful map $\Div_{n,d} \ra \Bun_{n,d}$ to study the motive of $\Bun_{n,d}$ in terms of that of $\Div_{n,d}$ and, in particular, to define an exhaustive sequence of vector bundles on $\Bun_{n,d}$. The definition of $\mu_{l}$ in the following theorem, especially the term $-\frac{1}{n^{2}}$, may look ad hoc, but this small rational correction term is required for the proof of Theorem \ref{thm1} below.

\begin{thm}\label{prop exh seq}
The stack $\Bun_{n,d}$ is exhaustive. More concretely, fix an effective divisor $D$ on $C$ and let $\mu_l := l \deg(D) -2g +1-\frac{1}{n^{2}}$ for $l \in \NN$; then there is an exhaustive sequence of vector bundles $(V_l \ra \Bun_{n,d}^{\leq \mu_l})_{l \in \NN}$ and
\[ M(\Bun_{n,d})  \simeq \colim_{l} R_{\tr}(\Div_{n,d}^{\leq \mu_l}(lD))\]
where $\Div_{n,d}^{\leq \mu_l}(lD):=\{E \hookrightarrow \cO_C(D)^{\oplus n} : \mu_{\max}(E) \leq \mu_l \} $ is the open subvariety of $\Div_{n,d}(lD)$ consisting of rank $n$ degree $d$ matrix divisors $E \hookrightarrow \cO_C(D)^{\oplus n}$ with $\mu_{\max}(E) \leq \mu_l$.
\end{thm}
\begin{proof}
For all vector bundles $E$ with $\mu_{\max}(E) \leq \mu_l$, as $\mu_{\max}(E)< \deg(lD)-2g +2$, we have
\[ H^1(E^\vee \otimes \cO_C(lD)^{\oplus n}) = 0.\]
Indeed if this vector space was non-zero then by Serre duality, there would exist a non-zero homomorphism $\cO_C(lD)^{\oplus n} \otimes \omega_C^{-1} \ra E$, but one can check that this is not possible by using the Harder-Narasimhan filtration of $E$ and standard results about homomorphisms between semistable bundles of prescribed slopes (see \cite[Proposition 1.2.7]{HL} and also \cite[\S 8.1]{bgl}). Hence, there is a vector bundle 
$V_l:=R^0p_{1*}(\cE^\vee_{\mathrm{univ}}\otimes p_2^*\cO_C(lD)^{\oplus n})$ over $\Bun_{n,d}^{\leq \mu_l}$, whose fibre over $E$ is the space $\Hom(E,\cO_C(lD)^{\oplus n})$. Let $U_l \subset V_l$ denote the open subset of injective homomorphisms; then
\[U_l \cong \Div^{\leq \mu_l}_{n,d}(lD).\]
We denote the closed complement of non-injective homomorphisms by $W_l \subset V_l$ and let $p_l : V_l \ra \Bun_{n,d}^{\leq \mu_l}$ denote the natural projection. 

Let us find a lower bound for $\codim_{V_{l}}(W_{l})$. Let us write $W_{l,\tau}:=p_{l}^{-1}(\Bun_{n,d}^{\tau})$ for an Harder-Narasimhan (HN) type $\tau$ with $\mu_{\max}(\tau) \leq \mu_l$. We have
\[
\codim_{V_{l}}(W_{l})\geq \min_{\tau: \mu_{\max}(\tau)\leq \mu_{l}}\codim_{V_{l}}(W_{l,\tau}).
  \]
By Lemma \ref{codim bgl correct} below, we have for each HN type $\tau$ with $\mu_{\max}(\tau) \leq \mu_l$
\[
\codim_{V_{l}}(W_{l,\tau})\geq \codim_{\Bun_{n,d}}(\Bun^{\tau}_{n,d})+ l\deg(D)-n\max \{h^{1}(E^{\vee})|\ E\in \Bun^{\tau}_{n,d}\}.
  \]
By Lemma \ref{codim brill noether} below, we deduce that, for an integer $K\geq 0$ depending only on $n,d$ and $g$, we have
  \[
\codim_{V_{l}}(W_{l})\geq l\deg(D)-K
    \]
and thus $\codim_{V_{l}}(W_{l})$ tends to infinity as $l$ tends to infinity.

 Let \[i_l: \Bun_{n,d}^{\leq \mu_l} \hookrightarrow \Bun_{n,d}^{\leq \mu_{l+1}}\] 
denote the open immersion; then the injective sheaf homomorphism $\cO_C(lD) \hookrightarrow \cO_C((l+1)D)$ determines an injective homomorphism $f_l : V_l \ra i_l^*V_{l+1}$. Moreover $f_l^{-1}(W_{l+1}) \subset W_l$, as if $E \ra \cO_C(lD)^{\oplus n} \hookrightarrow \cO_C((l+1)D)^{\oplus n}$ is not injective, then the homomorphism $E \ra \cO_C(lD)^{\oplus n}$ is not injective. In particular, $(V_l,W_l)_{l \in \NN}$ is an exhaustive sequence of vector bundles over $\Bun_{n,d} = \cup_{l \in \NN} \Bun_{n,d}^{\leq \mu_l}$ and so we have
\[ M(\Bun_{n,d}) = \colim_{l} R_{\tr}(U_l) \simeq \colim_{l} R_{\tr}(\Div_{n,d}^{\leq \mu_l}(lD)) \]
as required.
\end{proof}

It remains to prove Lemmas \ref{codim bgl correct} and \ref{codim brill noether}. Lemma \ref{codim bgl correct} is a slight modification of \cite[Lemma 8.2]{bgl} with weaker assumptions. In fact, in the proof of \cite[Proposition 8.1]{bgl}, the stronger assumptions of \cite[Lemma 8.2]{bgl} are not satisfied when this lemma is applied. Furthermore, this lemma is incorrectly quoted and applied in \cite{BD}. However, one can instead apply the following lemma and then use the argument given in Proposition \ref{prop exh seq} above. 

\begin{lemma}\label{codim bgl correct}
Let $E$ and $F$ be rank $n$ locally free $\mathcal{O}_C$-modules and $D$ be an effective divisor on $C$ such that $\Ext^1(E,F(D)) = 0$. Then the codimension $c_{D,E,F}$ of the closed subset of non-injective sheaf homomorphisms in $\Hom(E,F(D))$ satisfies $c_{D,E,F} \geq \deg(D) - \dim \Ext^1(E,F)$.
\end{lemma}
\begin{proof}
Since $\Ext^1(E,F(D)) = 0$, we have an exact sequence of $k$-vector spaces
  \[
0\ra \Hom(E,F) \ra \Hom(E,F(D))\stackrel{\Phi}{\ra} \Hom(E,\mathcal{O}_D^{\oplus n})\ra \Ext^1(E,F) \ra 0
\]
and thus
\begin{equation}\label{coconut}
\dim \Hom(E,\cO_D^{\oplus n}) - \dim \Hom(E,F(D))+\chi(E,F) = 0.
\end{equation} 
Let $\Hom_{\text{ni}}(E,F(D))\subset \Hom(E,F(D))$ be the closed subset of non-injective homomorphisms, and $\Hom_{<n}(E,\cO_D^{\oplus n})\subset \Hom(E,\cO_D^{\oplus n})$ be the closed subset of homomorphisms which factor through a surjection $E\ra G\subset \cO_D^{\oplus n}$ such that for every $y\in \text{supp}(D)$, the rank of $G$ at $y$ is strictly less than $n$. As in the proof of \cite[Lemma 8.2]{bgl}, we have that the codimension $c'_{D,E}$ of $\Hom_{<n}(E,\cO_D^{\oplus n})$ in $\Hom(E,\cO_D^{\oplus n})$ satisfies
\begin{equation}\label{guava}
  c'_{D,E}\geq \deg(D)
\end{equation}
and $\Phi(\Hom_{\text{ni}}(E,F(D)))\subset\Hom_{<n}(E,\cO_D^{\oplus n})$. Hence,
\begin{equation}\label{nispero}
\dim \Hom_{\text{ni}}(E,F(D))\leq \dim \Hom_{<n}(E,\cO_D^{\oplus n}) + \dim \Hom(E,F).
\end{equation}
For a closed subset $Y$ of a smooth irreducible variety $X$, we have $\codim_X(Y) = \dim(X) - \dim(Y)$. Thus
\begin{eqnarray*}
  c_{D,E,F} & \geq & \dim \Hom(E,F(D)) - \dim \Hom_{<n}(E,\cO_D^{\oplus n}) - \dim \Hom(E,F) \\
  & = & c'_{D,E}  - \dim \Ext^1(E,F) \\
  & \geq & \deg(D) - \dim \Ext^1(E,F),
\end{eqnarray*}
where the first inequality follows from \eqref{nispero}, the second equality follows from \eqref{coconut} and the final inequality follows from \eqref{guava}. This concludes the proof.
\end{proof}  

\begin{lemma}\label{codim brill noether}
There is an integer $K$, which depends only on $n,d$ and $g$, such that for any Harder-Narasimhan (HN) type $\tau$ of rank $n$ and degree $d$, we have
  \[
\codim_{\Bun_{n,d}}(\Bun^{\tau}_{n,d})-n\max \{h^{1}(E^{\vee})|\ E\in \Bun^{\tau}_{n,d} \} \geq K.
    \]
\end{lemma}  
\begin{proof}
Let $\tau = ((n_1,d_1), \dots, (n_r,d_r))$ record the ranks and degrees of the subquotients with $\sum_{i=1}^r n_i = n$ and $\sum_{i=1}^r d_i =d$. Taking the associated graded bundle defines a map
  \[
\mathrm{gr} : \Bun^{\tau}_{n,d}\ra \prod_{i=1}^{r} \Bun^{ss}_{d_{i},n_{i}}
    \]
which is an affine bundle of rank $\sum_{i<j} n_{i}n_{j}(g-1)+d_{j}n_{i}-d_{i}n_{j}$ by a Riemann-Roch computation (for instance, see \cite[Lecture 5]{Heinloth-lectures}). Since $\dim(\Bun^{ss}_{d_{i},n_{i}})=n_{i}^{2}(g-1)$, it follows that
\[ \codim_{\Bun_{n,d}}(\Bun_{n,d}^\tau) = (n^2 - \sum_{1 \leq i \leq j \leq r} n_i n_j)(g-1) + \sum_{i < j} (n_jd_{i} - n_{i}d_{j}), \]
where as $\tau$ is a HN type, we have $n_{j}d_{i}-n_{i}d_{j} \geq 0$ for all $i\leq j$. The first term can be bounded below by a constant which depends only on $n$ and $g$.

Let $E\in \Bun^{\tau}_{n,d}$ and $E_{1},\ldots,E_{r}$ be the subquotients of its HN filtration, with $\rk(E_{i})=n_{i}$ and $\deg(E_{i})=d_{i}$. By Serre duality, we have $h^{1}(E^{\vee})=h^{0}(E\otimes\omega_{C})$. We have
\[
h^{0}(E\otimes\omega_{C})\leq \sum_{i=1}^{r} h^{0}(E_{i}\otimes \omega_{C}).
  \]
  Let $1\leq i_{0}\leq i_{1}\leq r$ be such that $i\leq i_{0}\Leftrightarrow d_{i}>0$ and $i\leq i_{1}\Leftrightarrow\frac{d_{i}}{n_{i}}\geq -(2g-2)$.  For $i_{0}< i\leq i_{1}$, the vector bundle $E_{i} \otimes \omega_{C}$ is semistable of slope in $[0,2g-2]$, so that by Clifford's theorem for vector bundles \cite[Theorem 2.1]{Newstead-geography}, we have
  \[
h^{0}(E_{i}\otimes\omega_{C})\leq n_{i}+\frac{d_{i}}{2}+g-1\leq n+g-1.
\]
For $i> i_{1}$, the vector bundle $E_{i}\otimes \omega_{C}$ is semistable of negative slope, thus $h^{0}(E_{i}\otimes \omega_{C})=0$. For $i\leq i_{0}$, we have $h^{1}(E_{i}\otimes \omega_{C})= \dim \Hom(E_{i},\mathcal{O}_{C})=0$ by Serre duality and the fact that $E$ is semistable of positive slope. In this case, by Riemann-Roch
\[
h^{0}(E_{i}\otimes\omega_{C})= \chi(E_{i}\otimes \omega_{C})= n_{i}(g-1)+d_{i}.
  \]
We conclude that $\sum_{i > i_0} h^{0}(E_{i}\otimes \omega_{C})$ and $\sum_{i \leq i_0} h^{0}(E_{i}\otimes\omega_{C})-d_{i} =\sum_{i\leq i_{0}} n_i(g-1)$ are bounded above by constants which only depend on $n$ and $g$.
    
We are thus reduced to find a lower bound for
\[
\sum_{i < j} (n_jd_{i} - n_{i}d_{j}) -n\sum_{i\leq i_{0}}d_{i}.
  \]
  We have
  \[
\sum_{i < j} (n_jd_{i} - n_{i}d_{j}) = \sum_{i_{0}<i < j} (n_jd_{i} - n_{i}d_{j}) + \sum_{i\leq i_{0} < j} (n_jd_{i} - n_{i}d_{j}) + \sum_{i < j\leq i_{0}} (n_jd_{i} - n_{i}d_{j})
\]
and
\[
  n\sum_{i\leq i_{0}}d_{i} = \sum_{i\leq i_{0}} \sum_{j}n_{j}d_{i} = \sum_{i<j\leq i_{0}} n_{i}d_{j}+ \sum_{i\leq i_{0}<j} n_{j}d_{i}+ \sum_{i<j\leq i_{0}}n_{j}d_{i} + \sum_{i\leq i_{0}}n_{i}d_{i}.
  \]
We also observe that the last term in this expression can be written as
\[
  \sum_{i\leq i_{0}}n_{i}d_{i}  =  \sum_{i\leq i_{0}}n_{i}\left(d-\sum_{j\neq i}d_{j}\right) = d\sum_{i\leq i_{0}}n_{i} -\sum_{i<j\leq i_{0}} n_{j}d_{i} - \sum_{i\leq i_{0}<j}n_{i}d_{j}- \sum_{i<j\leq i_{0}} n_{i}d_{j}.
\]
Finally, we obtain
\[
  \sum_{i < j} (n_jd_{i} - n_{i}d_{j}) -n\sum_{i\leq i_{0}}d_{i} = \sum_{i_{0}< i<j}(n_{j}d_{i}-n_{i}d_{j})  + \sum_{i<j\leq i_{0}}(n_{j}d_{i}-n_{i}d_{j}) -d\sum_{i\leq i_{0}}n_{i}  \geq -nd,
\]
since $n_{j}d_{i}-n_{i}d_{j} >0$ for $i < j$, as $\tau$ is a HN type. This concludes the proof.
\end{proof}

\begin{thm}\label{thm1}
In $\DM(k,R)$, for any non-zero effective divisor $D$ on $C$, we have 
\[M(\Bun_{n,d}) \simeq \colim_{l} R_{\tr}(\Div_{n,d}(lD))\simeq \hocolim_{l} M(\Div_{n,d}(lD). \]
\end{thm}
\begin{proof}
The right isomorphism follows from Lemma~\ref{colim_hocolim}. By Proposition \ref{prop exh seq}, we have
\[M(\Bun_{n,d}) \simeq \colim_{l} R_{\tr}(\Div_{n,d}^{\leq \mu_l}(lD)). \]
We then obtain the left isomorphism by applying Proposition \ref{prop hocolim vanishes} to the inductive system of open immersions $(\Div_{n,d}^{\leq \mu_l}(lD) \hookrightarrow \Div_{n,d}(lD))_{l \in \NN}$. To apply this corollary, we need to check that the closed complements $\Div_{n,d}(lD) - \Div_{n,d}^{\leq \mu_l}(lD)$ have codimensions tending to infinity with $l$. Let $E$ be a vector bundle with $\mu_{\max}(E)>\mu_l$; then  $\mu_{\max}(E)+2g-1\geq \mu_l+2g-1=\deg(lD)$, and so by \cite[Proposition 5.2 (4)]{bgl}, we have
\[ \codim_{\Div_{n,d}(lD)}(\Div_{n,d}(lD) - \Div_{n,d}^{\leq \mu_l}(lD)) \geq l \deg D -c\]
for a constant $c$ independent of $l$, which completes the proof.
\end{proof}

We recall that a motive is pure if it lies in the heart of Bondarko's Chow weight structure on $\DM(k,R)$ defined in \cite{Bondarko_weight}. In particular, the motive of any smooth projective variety is pure and, as $M(\Bun_{n,d})$ is described as a homotopy colimit of motives of smooth projective varieties, we deduce the following result.

\begin{cor}
The motive $M(\Bun_{n,d})$ is pure.
\end{cor}

This corollary sits well with the fact that the cohomology of $\Bun_{n,d}$, and more generally the cohomology of moduli stacks of principal bundles on curves, is known to be pure in various contexts; for instance, if $k=\mathbb{C}$, the Hodge structure is pure by \cite[Proposition 4.4]{teleman}, and over a finite field, the $\ell$-adic cohomology is pure by \cite[Corollary 3.3.2]{hs}.

\subsection{The Bia{\l}ynicki-Birula decomposition for matrix divisors}

For a $\GG_m$-action on a smooth projective $k$-variety $X$, we recall the associated Bia{\l}ynicki-Birula decomposition (\cite{BB_original} when $k$ is algebraically closed, \cite{hesselink} when $k$ is arbitrary). Let $\{X_i\}_{i \in I}$ be the connected components of $X^{\GG_m}$ and let $X_i^+:=\{ x \in X: \lim_{t\ra 0} t\cdot x\in X_i \}$ denote the attracting set of $X_i$. Then $X=\coprod_{i\in I} X^+_i$ and $X^+_i$ is a smooth locally closed subset of $X$ with a retraction $X^+_i\ra X_i$ which is a Zariski locally trivial affine fibration. Moreover, the cohomology (and the motive) of $X$ can be described in terms of that of the fixed locus.

The Quot scheme $\Div_{n,d}(D)$ is a smooth projective variety of dimension $n^2 \deg D -nd$. The group $\GL_n$ acts on $\Div_{n,d}(D)$ by automorphisms of $\cO_C(D)^{\oplus n}$. If we fix a generic 1-parameter subgroup $\GG_m \subset \GL_n$ of the diagonal maximal torus $T=\GG_m^n$, then the fixed points of this $\GG_m$-action agree with the fixed points for the $T$-action. These actions and their fixed points were studied by Str{\o}mme \cite{stromme}; the fixed points are matrix divisors of the form
\[ \bigoplus_{i=1}^n \cO_C(D - F_i) \hookrightarrow \cO_C(D)^{\oplus n}\]
for effective divisors $F_i$ such that $\sum_{i=1}^n \deg F_i = n \deg D - d$. By specifying the degree $m_i$ of each $F_i$ we index the connected components of this torus fixed locus; more precisely, the components indexed by a partition $\underline{m}=(m_1, \dots , m_n)$ of $n \deg D - d$ is the following product of symmetric powers of $C$
\[ C^{(\underline{m})} := C^{(m_1)} \times \cdots \times  C^{(m_n)}. \]
Str{\o}mme also studied the associated Bia{\l}ynicki-Birula decomposition (for $C = \PP^1$) and this was later used by Bifet, Ghione and Letizia \cite{bgl} (for $C$ of arbitrary genus) to study the cohomology of moduli spaces of vector bundles. Using the same ideas, del Ba\~{n}o showed that the Chow motive of $\Div_{n,d}(D)$ (with $\QQ$-coefficients) is the $(n \deg D -d)$-th symmetric power of the motive of $C \times \PP^{n-1}$; see \cite[Theorem 4.2]{Del_Bano_motives_moduli}.

In order to define a Bia{\l}ynicki-Birula decomposition of $\Div_{n,d}(D)$, we fix $\GG_m \hookrightarrow \GL_n$ of the form $t \mapsto \diag(t^{w_1}, \dots , t^{w_n})$ with decreasing integral weights $w_1 > \cdots > w_n$. The action of $t \in \GG_m$ on $\Div_{n,d}(D)$ is given by precomposition with the corresponding automorphism of $\cO_C(D)^{\oplus n}$. The Bia{\l}ynicki-Birula decomposition for this $\GG_m$-action gives a stratification of 
\begin{equation}
 \Div_{n,d}(D) = \bigsqcup_{\underline{m} \dashv \: n\deg D - d} \Div_{n,d}(D)_{\underline{m}}^+ 
\end{equation}
where $\Div_{n,d}(D)_{\underline{m}}^+$ is a smooth locally closed subvariety of $\Div_{n,d}(D)$ consisting of points whose limit as $t \ra 0$ under the $\GG_m$-action lies in $C^{(\underline{m})}$. The stratum $\Div_{n,d}(D)_{\underline{m}}^+$ has codimension $c_{\underline{m}}^+:=\sum_{i=1}^n (i-1)m_i$ (\textit{cf.}\ \cite[$\S$3]{bgl} and \cite[$\S$6]{BD}).

The motivic Bia{\l}ynicki-Birula decomposition \cite{Brosnan, Choudhury_Skowera, Karpenko} in this case gives the following result.

\begin{cor}\label{cor mot BB decomposition}
In $\DM(k,R)$, we have a direct sum decomposition
\[
M(\Div_{n,d}(D))\simeq \bigoplus_{\underline{m}\dashv\: n\deg(D)-d} M(C^{(\underline{m})})\{c_{\underline{m}}^{+}\}.
\]
\end{cor}

\subsection{Properties of the motive of the stack of bundles} From the above results, we deduce the following properties of $M(\Bun_{n,d})$.

\begin{thm}\label{main Thm1}\
\begin{enumerate}
\item If $C = \PP^1$, then $M(\Bun_{n,d})$ is an (infinite dimensional) Tate motive.
\item If $R$ is a $\QQ$-algebra, then $M(\Bun_{n,d})$ is contained in the smallest localising tensor triangulated category of $\DM(k,R)$ containing the motive of the curve $C$. In particular, $M(\Bun_{n,d})$ is an abelian motive.
\end{enumerate}
\end{thm}
\begin{proof} 
This result follows from Theorem \ref{thm1} and Corollary \ref{cor mot BB decomposition}. For the first statement, we use the fact that symmetric powers of $\PP^1$ are projective spaces. For the second statement, as $R$ is a $\QQ$-algebra, for any motive $M\in \DM(k,R)$ and $i\in\NN$, there exists a symmetric power $\Sym^i M$ which is a direct factor of $M^{\otimes i}$ and such that, for any $X$ smooth quasi-projective variety, we have $M(X^{(i)})\simeq \Sym^i M(X)$ \cite[Proposition 2.4]{nilp}. The final claim follows as the motive of a curve is an abelian motive (that is, it lies in the localising subcategory generated by motives of abelian varieties) \cite[Proposition 4.2.5, Lemma 4.3.2]{Ancona_Huber} and abelian motives are preserved under tensor product.
\end{proof}

A similar result was obtained by del Ba\~{n}o for the motive of the moduli \emph{space} of stable vector bundles of fixed rank and degree \cite[Theorem 4.5]{Del_Bano_motives_moduli}.

\subsection{A conjecture on the transition maps}
\label{sec:conj-trans-maps}

In order to obtain a formula for this motive, we need to understand the functoriality of the motivic BB decompositions for the closed immersions $i_{D,D'}: \Div_{n,d}(D) \ra \Div_{n,d}(D')$ for divisors $D' \geq D \geq 0$. More precisely, the map $i_{D,D'}$ and the decompositions of Corollary~\ref{cor mot BB decomposition} induce a commutative diagram
\begin{equation}\label{martragny}
  \xymatrix@C=2cm{
    M(\Div_{n,d}(D)) \ar[r]^{M(i_{D,D'})} \ar[d]_{\wr} & M(\Div_{n,d}(D')) \ar[d]_{\wr} \\
    \bigoplus\limits_{\underline{m}\dashv \:n\deg(D)-d} M(C^{(\underline{m})})\{c_{\underline{m}}\} \ar[r]^{\bigoplus k_{\underline{m},\underline{m}'}} & \bigoplus\limits_{\underline{m}'\dashv \:n\deg(D')-d}M(C^{(\underline{m}')})\{c_{\underline{m}'}\}
}
\end{equation}
with induced morphisms $k_{\underline{m},\underline{m}'}:M(C^{(\underline{m})})\{c_{\underline{m}}\}\ra M(C^{(\underline{m}')})\{c_{\underline{m}'}\}$ between the factors.

Although we have $i_{D,D'}(C^{(\underline{m})}) \subset C^{(\underline{m}+ \underline{\delta})}$ for $\delta :=\deg(D' -D)$, the morphisms  $k_{\underline{m},\underline{m}'}$ are not induced by these fixed loci inclusions, as the closed subscheme $\Div_{n,d}(D) \hookrightarrow \Div_{n,d}(D')$ does not intersect the BB strata in $\Div_{n,d}(D')$ transversally. Indeed, we note that $c_{\underline{m}} \neq c_{\underline{m}+\underline{\delta}}$. 

However, we have $c_{\underline{m}} = c_{\underline{m}'}$, when $\underline{m}' = \underline{m} + (n\delta,0,\dots,0)$. In this case, there is a morphism
\[ f_{\underline{m},\underline{m'}}:=a_{D'-D} \times \id_{C^{(m_2)}} \times \cdots \times  \id_{C^{(m_n)}} : C^{(\underline{m})} \ra C^{(\underline{m}')}\]
where $a_{D'-D}:C^{(m_1)}\ra C^{(m'_1)}$ corresponds to adding $n(D'-D)$. Based on some small computations, we make the following conjectural description for the morphisms $k_{\underline{m},\underline{m}'}$.

\begin{conj}\label{thm2}
Let $D'\geq D$ be effective divisors on $C$; then the morphisms 
\[k_{\underline{m},\underline{m}'}:M(C^{(\underline{m})})\{c_{\underline{m}}\}\ra M(C^{(\underline{m}')})\{c_{\underline{m}'}\}\] 
fitting into the commutative diagram \eqref{martragny} are the morphisms $M(f_{\underline{m},\underline{m'}})\{c_{\underline{m}} \}$ when $\underline{m}' = \underline{m} + (n\deg(D'-D),0,\dots,0)$ and are zero otherwise.
\end{conj}

One can also formulate this on the level of Chow groups, using the BB decomposition for Chow groups, which leads to a conjecture on the intersection theory of these Quot schemes.

\begin{conj}\label{conj int thy quot}
For effective divisors $D'\geq D$ on $C$, we have a commutative diagram
\begin{equation}
  \xymatrix@C=2cm{
    \CH^*(\Div_{n,d}(D')) \ar[r]^{i^*_{D,D'}} \ar[d]_{\wr} & \CH^{*}(\Div_{n,d}(D)) \ar[d]_{\wr} \\
    \bigoplus\limits_{\underline{m'}\dashv \:n\deg(D')-d} \CH^{*-c_{\underline{m}'}}(C^{(\underline{m}')}) \ar[r]^{\bigoplus k_{\underline{m},\underline{m}'}} & \bigoplus\limits_{\underline{m}\dashv \:n\deg(D)-d}\CH^{*-c_{\underline{m}}}(C^{(\underline{m})})
}
\end{equation}
where the vertical maps are the BB isomorphisms and $k_{\underline{m},\underline{m}'} :  \CH^{*-c_{\underline{m}'}}(C^{(\underline{m}')}) \ra  \CH^{*-c_{\underline{m}}}(C^{(\underline{m})})$ is given by $ f_{\underline{m},\underline{m'}}^*$ when $\underline{m}' = \underline{m} + (n\deg(D'-D),0,\dots,0)$ and is zero otherwise.
\end{conj}

\begin{rmk}\label{rmk conj int thy implies conj motives}
Since $\Div_{n,d}(D)$ are smooth varieties, their motives encode their Chow groups. In particular, this means Conjecture \ref{thm2} implies Conjecture \ref{conj int thy quot}. In fact, at least with rational coefficients, Conjecture \ref{conj int thy quot} for all field extensions of $k$ is equivalent to Conjecture \ref{thm2} for all field extensions of $k$ by \cite[Lemma 1.1]{Huybrechts_K3}.
\end{rmk}

In Theorem \ref{thm3} below, we deduce a conjectural formula for $M(\Bun_{n,d})$ (\textit{cf.}\ Conjecture \ref{conj formula})  from Conjecture \ref{thm2}. 

\subsection{The motive of the stack of line bundles}

Throughout this section, we assume that $C(k) \neq \emptyset$ and we prove a formula for the motive of the stack of line bundles (\textit{cf.}\ Corollary \ref{cor bun_1}) by proving a result about the motive of an inductive system of symmetric powers of curves (\textit{cf.}\ Lemma \ref{lemma hocolim sym}); the proof of this lemma is probably well-known, at least in cohomology, but we nevertheless include a proof as a generalisation of this argument is used in Theorem \ref{thm3}.

\begin{lemma}\label{lemma hocolim sym}
Suppose that $C(k) \neq \emptyset$. For $d \in \ZZ$ and an effective divisor $D_0$ on $C$ of degree $d_0>0$, consider the inductive system $(C^{(l d_0-d)})_{l\geq |d|}$ with $a_{D_0}: C^{(ld_0-d)} \ra C^{((l+1)d_0-d)}$ given by sending a degree $ld_0-d$ effective divisor $D$ to $D+D_0$. Then in $\DM(k,R)$ we have
\[ \hocolim_{l\geq |d|} M(C^{(ld_0-d)}) \simeq M(\Jac(C)) \otimes M(B\GG_m). \]
\end{lemma}
\begin{proof}
For notational simplicity, we prove the statement for $d = 0$; the proof is the same in general. We consider the Abel-Jacobi maps $AJ_l : C^{(ld_0)} \ra \Jac(C)$ defined using $lD_0$ which are compatible with the morphisms $a_{D_0}: C^{(ld_0)} \ra C^{((l+1)d_0)}$. For $ld_0 > 2g-2$ as $C(k) \neq \emptyset$, the Abel-Jacobi map is a  $\PP^{ld_0-g}$-bundle: we have $C^{(ld_0)} \cong \PP(p_*\cP_l)$, where $\cP_l$ is a Poincar\'{e} bundle on $\Jac(C) \times C$ of degree $ld_0$ and $p: \Jac(C) \times C \ra \Jac(C)$ is the projection. In fact, we can assume that $\cP_{l+1} = \cP_l \otimes q^*(\cO_C(D_0))$ for the projection $q: \Jac(C) \times C \ra C$. Then $\cP_l$ is a subbundle of $\cP_{l+1}$, and the induced map between the projectivisations is $a_{D_0}$.

By the projective bundle formula, for $ld_0 > 2g-2$, we have 
\[ M(C^{(ld_0)}) \simeq M(\PP^{ld_0-g}) \otimes M(\Jac(C))\]
such that the transition maps $M(C^{(ld_0)}) \ra M(C^{((l+1)d_0)})$ induce the identity on $M(\Jac(C))$. Hence, by Lemma \ref{pulling out constants hocolims}, we can pull out the motive of $\Jac(C)$ from this homotopy colimit 
\[ \hocolim_l M(C^{(ld_0)}) \simeq  M(\Jac(C)) \otimes \hocolim_l M(\PP^{ld_0-g}). \]
By Example~\ref{ex mot BGm}, we have $M(B\GG_m) \simeq \hocolim_r M(\PP^r)$. As the inductive system $(M(\PP^{ld_0-g}))_l$ is a cofinal subsystem of this system, we conclude the result using \cite[Lemma 1.7.1]{neeman}.
\end{proof}

\begin{rmk}\label{rmk litt}
If $C(k) = \emptyset$, then the Abel-Jacobi map from a sufficiently high symmetric power of $C$ is not a projective bundle in general, but rather a Brauer-Severi bundle \cite{Litt}.
\end{rmk}

From this result we obtain a formula for the motive of the stack of line bundles.

\begin{cor}\label{cor bun_1}
Suppose that $C(k) \neq \emptyset$. Then in $\DM(k,R)$, there is an isomorphism
\[ M(\Bun_{1,d}) \simeq M(\Jac(C)) \otimes M(B\GG_m).\]
In particular, if $C$ has a rational point or $d = 0$, then $M(\Bun_{1,d}) \simeq M(\Jac(C)) \otimes M(B\GG_m)$.
\end{cor}
\begin{proof}
By Theorem \ref{thm1}, we have $M(\Bun_{1,d})  \simeq  \hocolim_l M(\Div_{1,d}(lD_0))$ for any effective divisor $D_0$ on $C$ of degree $d_0 >0$. Since $ \Div_{1,d}(lD_0) \cong C^{(ld_0 - d)}$, we have
\[
M(\Bun_{1,d}) \simeq  \hocolim_l M(\Div_{1,0}(lD_0)) \simeq  \hocolim_l M(C^{(ld_0-d)})\simeq M(\Jac(C)) \otimes M(B\GG_m)\]
by Lemma \ref{lemma hocolim sym}. 
\end{proof}

\begin{rmk}
Alternatively, we can deduce this result as $\Bun_{1,d} \ra \Pic^d(C) \cong \Jac(C)$ is a trivial $\GG_m$-gerbe (or equivalently, $\Pic^d(C)$ is a fine moduli space, which is true as by assumption $C(k) \neq \emptyset$ and thus the Poincar\'{e} bundle gives a universal family).
\end{rmk}

\subsection{A conjectural formula for the motive} Throughout this section, we continue to assume that $C(k) \neq \emptyset$ and we deduce a formula for $M(\Bun_{n,d})$ from Theorem \ref{thm1} and Conjecture \ref{thm2}. 

\begin{defn}
Let $X$ be a quasi-projective $k$-variety and $N\in \DM(k,R)$. The motivic zeta function of $X$ at $N$ is
\[
Z(X,N)=\bigoplus_{i=0}^{\infty} M(X^{(i)})\otimes N^{\otimes i}\in \DM(k,R).
\]
\end{defn}

We can now state our main conjecture; for evidence supporting this conjecture, see $\S$\ref{sec compare}.

\begin{conj}\label{conj formula}
Suppose that $C(k) \neq \emptyset$; then in $\DM(k,R)$, we have
\[M(\Bun_{n,d}) \simeq  M(\Jac(C)) \otimes M(B\GG_m) \otimes \bigotimes_{i=1}^{n-1} Z(C, R\{i\}). \]
\end{conj}

\begin{thm}\label{thm3}
Conjecture~\ref{thm2} implies Conjecture~\ref{conj formula}.
\end{thm}
\begin{proof}
By Theorem~\ref{thm1}, for any non-zero effective divisor $D_0$ on $C$, we have
\[M(\Bun_{n,d}) \simeq \hocolim_{l} M(\Div_{n,d}(lD_0)) \]
with transition morphisms induced by $i_l:\Div_{n,d}(lD_0)\ra \Div_{n,d}((l+1)D_0))$. Let $d_0=\deg(D_0)$; then  the decomposition from Corollary \ref{cor mot BB decomposition} for the divisor $D=lD_0$ is
\[
M(\Div_{n,d}(lD_0))\simeq \bigoplus_{\underline{m}\dashv \:nld_0-d} M(C^{(\underline{m})})\{c_{\underline{m}}\}
\]
We can write the inductive system $l\mapsto M(\Div_{n,d}(lD_0))$ as a direct sum of inductive systems as follows. For $\underline{m}^\flat= (m^\flat_2,\ldots,m^{\flat}_n)\in \NN^{n-1}$ and $l \in \NN$, we write $m_1(l):=nld_0-d-\sum_{i=2}^{n-1} m^\flat_i$ and $\underline{m}(l):=(m_1(l),\underline{m}^{\flat})\in \ZZ\times \NN^{n-1}$. Notice that $c_{\underline{m}(l)}=\sum_i (i-1) m^{\flat}_i$ only depends on $\underline{m}^\flat$, and we will also use the notation $c_{\underline{m}^\flat}$. For $\underline{m}^\flat \in \NN^{n-1}$, we define an inductive system $P_{\underline{m}^{\flat},*}: \NN\rightarrow \DM(k,R)$ as follows:
\[
P_{\underline{m}^{\flat},l}:= \left\{\begin{array}{ll} 0 & \text{if }\: m_1(l)<0, \\ M(C^{(\underline{m}(l))})\{c_{\underline{m}^\flat}\} & \text{if }\: m_1(l)\geq 0 \end{array}\right.
\]
where the map $P_{\underline{m}^\flat,l}\rightarrow P_{\underline{m}^\flat,l+1}$ is zero if $m_1(l+1)<0$ and the morphism
\[
k_{\underline{m}^{\flat},l}:=\left(M(a_{nD_0})\otimes \id_{M(C^{(\underline{m}^{\flat})})}\right)\{c_{\underline{m}^\flat}\}.
\]  
if $m_1(l+1)\geq 0$. Assuming Conjecture~\ref{thm2}, we have an isomorphism
\[
M(\Div_{n,d}(lD_0))\simeq \bigoplus_{\underline{m}^\flat\in \NN^{n-1}} P_{\underline{m}^\flat,l}
\]
as inductive systems of motives indexed by $l \in \NN$. By Corollary~\ref{cor mot BB decomposition} and Lemma~\ref{lemma hocolims sums}, we deduce 
\begin{equation}\label{arepa}
M(\Bun_{n,d})\simeq \bigoplus_{\underline{m}^\flat\in \NN^{n-1}} \hocolim_l P_{\underline{m}^\flat,l}.
\end{equation}
For each $\underline{m}^{\flat}$ and $l$, we have a generalised Abel-Jacobi map
\begin{equation}
AJ_{\underline{m}^{\flat},l} : C^{(\underline{m}(l))} \ra \Pic^{nld_0 - d}(C) \times C^{(m_2)} \times \cdots \times C^{(m_n)} \cong \Jac(C) \times C^{(\underline{m}^{\flat})}
\end{equation}
sending $(F_1, \dots, F_n)$ to $(\cO_C(\sum_{i=1}^n F_i), F_2, \dots , F_n)$. In fact, if $m_1(l) > 2g -2$, this morphism is a $\PP^{m_1(l)-g}$-bundle: we have that $C^{(\underline{m}(l))} \cong \PP(p_*\cF)$ where $p : \Jac(C) \times C^{(\underline{m}^{\flat})} \times C \ra \Jac(C) \times C^{(\underline{m}^{\flat})}$ is the projection and $\cF$ is the tensor product of the pullback of the degree $nld_0 - d$ Poincar\'{e} bundle $\cP \ra \Jac(C) \times C$ with the pullbacks of the duals of the universal line bundles $\cL_i \ra C^{(m_i)} \times C$ for $2 \leq i \leq n$. In this case, by the projective bundle formula, we have
\[ M(C^{(\underline{m}(l))}) \simeq M(\Jac(C)) \otimes M(C^{(\underline{m}^{\flat})}) \otimes M(\PP^{m_1(l)-g}) \]
Then by Lemma~\ref{pulling out constants hocolims} and Lemma~\ref{lemma hocolim sym}, we have
\begin{align*}
 \hocolim_l P_{\underline{m}^\flat,l} &\simeq  \hocolim_{l\: :\: m_1(l)>2g-2} M(C^{(\underline{m}(l))})\{c_{\underline{m}^\flat}\} \\ 
 & \simeq  \hocolim_{l \:: \: m_1(l)>2g-2} M(\PP^{m_1(l)-g}) \otimes M(\Jac(C)) \otimes M(C^{(\underline{m}^{\flat})}) \{c_{\underline{m}^\flat}\}   \\
 & \simeq  M(B\GG_m) \otimes M(\Jac(C)) \otimes M(C^{(\underline{m}^{\flat})})\{c_{\underline{m}^\flat}\}.
\end{align*}

Hence, using the commutation of sums and tensor products, we have
\begin{eqnarray*}
  M(\Bun_{n,d}) &\simeq & \bigoplus_{\underline{m}^\flat\in \NN^{n-1}} M(B\GG_m) \otimes M(\Jac(C))\otimes M(C^{({\underline{m}^\flat})})\{c_{\underline{m}^{\flat}}\}\\
  & \simeq &   M(B\GG_m)\otimes M(\Jac(C)) \otimes \bigoplus_{\underline{m}^\flat\in \NN^{n-1}}\bigotimes_{i=2}^{n} M(C^{(m^\flat_i)})\{(i-1)m^{\flat}_i\}\\
  & \simeq & M(B\GG_m)\otimes  M(\Jac(C)) \otimes \bigotimes_{i=1}^{n-1} Z(C,R\{i\}), 
\end{eqnarray*}
which completes the proof of the theorem.
\end{proof}

\begin{rmk}
Since the morphism $\det:\Bun_{n,d}\ra \Bun_{1,d}$ is not representable or of finite type, one needs to carefully define the induced morphism of motives $M(\det):M(\Bun_{n,d})\ra M(\Bun_{1,d})$; we do not present this additional construction here. However, modulo this extra work, the proof of Theorem \ref{thm3} implies the following compatibility between the formula of Conjecture \ref{conj formula} and the isomorphism $M(\Bun_{1,d})\simeq M(\Jac(C))\otimes M(B\GG_{m})$ of Corollary \ref{cor bun_1}  First, note that there is a canonical direct factor $R\{0\}$ in the motive $\bigotimes_{i=1}^{n-1} Z(C,R\{i\})$ and thus a direct factor $M(\Jac(C))\otimes M(B\GG_{m})$ in the formula of Conjecture~\ref{conj formula}. Then the map $M(\det)$ is precisely the projection onto that direct factor.
\end{rmk}

\section{Consequences and comparisons with previous results}
\label{sec conseq and comp}

As above, we let $\Bun_{n,d}$ denote the stack of vector bundlesof rank $n$ and degree $d$ over a smooth projective geometrically connected curve $C$ of genus $g$ over a field $k$.

\subsection{The compactly supported motive}

Let $M^c(\Bun_{n,d}) \in \DM(k,R)$ denote the compactly supported motive of $\Bun_{n,d}$ (as defined in $\S$\ref{sec mot comp}). By Theorem \ref{thm1} and the fact that the dual of an infinite sum is an infinite product, we have for any fixed divisor $D_{0}>0$
\[
M^{c}(\Bun_{n,d})\{ -n^2(g-1)\}\simeq \holim_{l} M(\Div_{n,d}(lD_{0}))\{-(n^2l\deg(D_0) - nd)\}.
\]

\begin{thm}\label{compact supp}
Assume that Conjecture~\ref{conj formula} holds and that $C(k)\neq \emptyset$; then, in $\DM(k,R)$, we have
\[M^c(\Bun_{n,d})  \simeq M^c(B\GG_m)\{(n^2 -1 )(g-1) \} \otimes M^c(\Jac C)  \otimes \bigotimes_{i=2}^n Z(C, R\{-i\}). \]
\end{thm}
\begin{proof}
We apply Poincar\'{e} duality for smooth stacks (\textit{cf.}\ Proposition \ref{prop PD for stacks}) to the formula for $M(\Bun_{n,d})$ in Conjecture \ref{conj formula}. Let us first calculate the duals of the motives of the Jacobian of $C$ and of the classifying space $B\GG_m$, and of the motivic zeta function. As $\Jac(C)$ is smooth and projective of dimension $g$, we have by Poincar\'{e} duality
\[ M(\Jac(C))^\vee \simeq M^c(\Jac(C))\{ -g \} \simeq M(\Jac(C))\{ -g \}. \]
As $B\GG_m$ is a smooth quotient stack of dimension $-1$, we have by Proposition \ref{prop PD for stacks} that
\[ M(B\GG_m)^\vee \simeq M^c(B\GG_m)\{ 1\}.\]
As the dual of an infinite sum of motives is the infinite product of the dual motives, we have
\begin{align*}
 Z(C,R\{i\})^\vee &  =\left( \bigoplus_{j=0}^\infty M(C^{(j)})\{ij\} \right)^\vee = \prod_{j=0}^\infty \left(M(C^{(j)})\{ij\}\right)^\vee \simeq \prod_{j=0}^\infty M(C^{(j)})^\vee \{-ij\} \\ & \simeq  \prod_{j=0}^\infty M(C^{(j)})\{-j\} \{-ij\} =\prod_{j=0}^\infty M(C^{(j)}) \{-(i+1)j\},
\end{align*}
as the symmetric power $C^{(j)}$ of the curve $C$ is a smooth projective variety of dimension $j$.

By Corollary \ref{cor sum prod}, we have that for $l \geq 2$, the natural morphism in $\DM(k,R)$
\[ \bigoplus_{j=0}^\infty M(C^{(j)})\{-lj\} \ra \prod_{j=0}^\infty M(C^{(j)})\{-lj\} \]
is an isomorphism. Hence, $Z(C,R\{i\})^\vee \simeq Z(C,R\{-(i+1)\})$ for $i \geq 1$.

As $\Bun_{n,d}$ is a smooth stack of dimension $n^2(g-1)$, Poincar\'{e} duality gives
\begin{align*}
M^c(\Bun_{n,d}) & \simeq M(\Bun_{n,d})^\vee\{n^2(g-1) \} \\ & \simeq  M(B\GG_m)^\vee \otimes M(\Jac(C))^\vee \otimes \bigotimes_{i=1}^{n-1} Z(C, \QQ\{i\})^\vee\{n^2(g-1) \} \\ 
& \simeq M^c(B\GG_m)\{(n^2 -1 )(g-1) \} \otimes M^c(\Jac C)  \otimes \bigotimes_{i=2}^n Z(C, \QQ\{-i\})
\end{align*} 
which gives the above formula.
\end{proof}

\subsection{Comparison with previous results}\label{sec compare}

In this section, we compare our conjectural formula for the (compactly supported) motive of $\Bun_{n,d}$ (\textit{cf.}\ Theorem \ref{compact supp}) with other results concerning topological invariants of $\Bun_{n,d}$. One of the first formulae to appear in the literature, was a computation of the stacky point count of $\Bun_{n,d}$ over a finite field $\FF_q$, which is defined as
\[ |\Bun_{n,d}(\FF_q)|_{\text{st}}:= \sum_{E \in \Bun_{n,d}(\FF_q)} \frac{1}{|\Aut(E)|}. \]

\begin{thm}[Harder]\label{thm harder}
Over a finite field $\FF_q$, we have
\[  |\Bun_{n,d}(\FF_q)|_{\text{st}}= \frac{q^{(n^2-1)(g-1)}}{q-1} |\Jac(C)(\FF_q)| \prod_{i=2}^n \zeta_C(q^{-i}) \]
where $\zeta_C$ denotes the classical Zeta function of $C$.
\end{thm}

We note that our conjectural formula for $M_c(\Bun_{n,d})$ is a direct translation of this formula, when one replaces a variety (or stack) by its (stacky) point count, $q^k$ by $R\{k\}$, and the classical Zeta function by the motivic Zeta function. In fact, the $\ell$-adic realisation of Theorem \ref{compact supp} should imply Theorem \ref{thm harder} by using Behrend's Lefschetz trace formula for the stack $\Bun_{n,d}$; since this relies on a conjecture, we do not provide the details.

Behrend and Dhillon \cite{BD} give a conjectural description for the class of the stack of principal $G$-bundles over $C$ in a dimensional completion $\widehat{K}_0(\Var_k)$ of the Grothendieck ring of varieties when $G$ is a semisimple group and, moreover, they prove their formula for $G = \SL_{n}$ by following the geometric arguments in \cite{bgl}. In fact, in \cite{BD}, it is also implicitly assumed that $C$ has a rational point in order to use the same argument involving Abel-Jacobi maps. By a minor modification of their computation, one obtains the following formula for the class of $\Bun_{n,d}$.

\begin{thm}[Behrend--Dhillon]\label{thm BD formula}
In $\widehat{K}_0(\Var_k)$, the class of $\Bun_{n,d}$ is given by
\[ [\Bun_{n,d}] = \LL^{(n^2-1)(g-1)}[B\GG_m][\Jac(C)]\prod_{i=2}^n Z(C,\LL^{-i}) \]
where $\LL := [\AA^1]$ and $Z(C,t):=\sum_{j\geq 0} [C^{(j)}]t^j$.
\end{thm}

It is possible to compare their formula with our conjectural formula for $M^c(\Bun_{n,d})$ if one passes to the Grothendieck ring of a certain dimensional completion of $\DM(k,R)$. We sketch this comparison here. First note that it does not make sense to work with the Grothendieck ring of $\DM(k,R)$, as this category is cocomplete and so by the Eilenberg swindle, $K_0(\DM(k,R)) \simeq 0$. Instead, we will follow the lines of Zargar \cite[$\S$3]{Zargar} and work with a completion of $\DM(k,R)$; however, note that Zargar works with effective motives and completes with respect to the slice filtration and we will instead complete with respect to the dimensional filtration.

For this completion process, one would like to take limits of projective systems of triangulated categories, but the category of triangulated categories is unsuitable for this task; hence, in the rest of this section, we let $\DM(k,R)$ denote the symmetric monoidal stable $\infty$-category underlying Voevodsky's category (associated to the model category of motivic complexes discussed in Section 2) . As in Definition \ref{defn dim filtr}, given $m\in\ZZ$, we write $\DM_{\gm}(k,R)_{r}$ (resp.\ $\DM(k,R)_{r}$) for the full sub-$\infty$-category (resp.\ presentable sub-$\infty$-category) of $\DM(k,R)$ generated by motives of the form $M^{c}(X)(r)$ with $X$ being a separated finite type $k$-scheme with $\dim(X)+r\leq n$. The localisations $\DM_{(\gm)}(k,R)/\DM_{(\gm)}(k,R)_{r}$ are symmetric monoidal $\infty$-categories (presentable in the non-geometric case), as this filtration is symmetric monoidal. One can then define
\[
\DM^{\wedge}_{(\gm)}(k,R) := \lim_{r\in \NN} \DM_{(\gm)}(k,R)/\DM_{(\gm)}(k,R)_{-r}
  \]
  as symmetric monoidal $\infty$-categories. By localisation, the functor $M^{c}:\Var_{k}\ra\DM_{\gm}(k,R)$ induces a ring morphism
  \[\chi_{c}:\widehat{K}_{0}(\Var_{k})\ra K_{0}(\DM^{\wedge}_{\gm}(k,R)).\]

One can then adapt the argument in \cite[Lemma 3.2]{Zargar} to show that the two functors
\[
\DM_{\gm}(k,R)\ra \DM^{\wedge}_{\gm}(k,R) \ra \DM^{\wedge}(k,R)
\]
are fully faithful. There is also a functor $(-)^{\wedge}:\DM(k,R)\ra \DM^{\wedge}(k,R)$. The following result is then clear from comparing the two formulas and using that $M(C^{(j)})\{-ij\}$ lies in $\DM_{\gm}(k,R)_{j(1-i)}\subset \DM_{\gm}(k,R)_{-j}$ for $i\geq 2$.

\begin{lemma}\label{lemma comp BD}
  Assume that $C(k)\neq\emptyset$ and that Conjecture \ref{conj formula} holds. Then $M^{c}(\Bun_{n,d})^{\wedge}$ lies in $\DM^{\wedge}_{\gm}(k,R)$ and we have $\chi_{c}[\Bun_{n,d}]=[M^{c}(\Bun_{n,d})^{\wedge}]$ in $K_{0}(\DM^{\wedge}_{\gm}(k,R))$.
\end{lemma}

Note that, as in \cite{Zargar}, it is not clear how small $K_{0}(\DM^{\wedge}_{\gm}(k,R))$ is, and in particular if the natural map $K_{0}(\DM_{\gm}(k,R))\ra K_{0}(\DM^{\wedge}_{\gm}(k,R))$ is injective, which limits the interest of such a comparison. To go further, one could try to introduce an analogue of the ring $\cM(k,R)$ in \cite{Zargar}; however, we do not pursue this here.

\subsection{Vector bundles with fixed determinant and $\SL_n$-bundles}
\label{sec fixed det}
For a line bundle $L$ on $C$ of degree $d$, we let $\Bun_{n,d}^L$ denote the stack of rank $n$ degree $d$ vector bundles over $C$ with determinant isomorphic to $L$, and we let $\Bun_{n,d}^{\simeq L}$ be the stack of pairs $(E,\phi)$ with $E$ a rank $n$ degree $d$ vector bundle and $\phi:\det(E)\simeq L$. We have the following diagram of algebraic stacks with cartesian squares
\[
  \xymatrix{
\Bun^{\simeq L}_{n,d}  \ar[d] \ar[r] & \Spec(k) \ar[d] & \\
\Bun^L_{n,d} \ar[r] \ar[d] & B\GG_m \ar[r] \ar[d] & \Spec(k) \ar[d]_{L} \\
\Bun_{n,d} \ar[r]^{\det} & \Bun_{1,d} \ar[r] & \Pic^{d}(C)
  }
  \]
where the morphism $\det$ is smooth and surjective. In fact, the bottom left square is cartesian by the See-saw Theorem. We see that $\Bun^{L}_{n,d}$ is thus a closed smooth substack of $\Bun_{n,d}$ of codimension $g$, and that $\Bun^{\simeq L}_{n,d}\ra \Bun^{L}_{n,d}$ is a $\GG_m$-torsor. Moreover, when $d=0$ and $L=\cO_{C}$, the stack $\Bun_{n,0}^{\simeq \cO_C}$ is actually isomorphic to the stack $\Bun_{\SL_n}$ of principal $\SL_n$-bundles; indeed, the vector bundle associated to an $\SL_n$-bundle via the standard representation has its determinant bundle canonically trivialised. For more details on this picture, in the more general case of principal bundles, see \cite[\S 1-2]{BLS}.

We will use these facts to compute the motive of $\Bun_{\SL_n}$ assuming Conjecture~\ref{thm2}. First, we claim that the smooth stacks $\Bun^{L}_{n,d}$ and $\Bun^{\simeq L}_{n,d}$ are exhaustive; the claim for the latter follows from the claim for the former by Proposition \ref{funct rep quot stacks}, as the $\GG_m$-torsor $\Bun^{\simeq L}_{n,d}\ra \Bun^{L}_{n,d}$ is representable and flat of finite type. To prove that $\Bun^{L}_{n,d}$ is exhaustive, we use matrix divisors with fixed determinant $L$. For every effective divisor $D$ on $C$, let $\Div_{n,d}^L(D)$ be the subvariety of $\Div_{n,d}(D)$ parametrising matrix divisors with determinant $L$; this is a smooth closed subscheme of codimension $g$ (see \cite[\S 6]{Dhillon_tamagawa}). Then one can prove that $\Bun_{n,d}^L$ is exhaustive analogously to Proposition \ref{prop exh seq} by using matrix divisors with fixed determinant $L$; for the relevant codimension estimates, see \cite[\S 6]{BD}. 

\begin{prop}\label{conj for fixed det}
For an effective divisor $D$ on $C$, the closed immersion $j_D : \Div_{n,d}^L(D) \ra \Div_{n,d}(D)$ is transverse to the BB strata in $\Div_{n,d}(D)$. Hence, Conjecture~\ref{thm2} implies the analogous statement for the Quot schemes of matrix divisors with fixed determinant.
\end{prop}
\begin{proof}
The fixed loci for the $\GG_m$-action on $\Div_{n,d}^L(D)$ are smooth closed subvarieties of $C^{(\underline{m})}$, which we denote by $C^{(\underline{m})}_L$ and consist of $(D_1, \dots , D_n) \in C^{(\underline{m})}$ such that $nD - \sum_{i=1}^n D_i \in  |L|$. It follows from \cite[\S 6]{BD} that
\[ \codim_{\Div_{n,d}^L(D)}(\Div_{n,d}^L(D)_{\underline{m}}^+) =  \codim_{\Div_{n,d}(D)}(\Div_{n,d}(D)_{\underline{m}}^+) =: c_{\underline{m}}^+. \]
Hence, it suffices to verify that the pullbacks of the normal bundles of $\Div^L_{n,d}(D) \hookrightarrow \Div_{n,d}(D)$ and $\Div_{n,d}(D)_{\underline{m}}^+ \hookrightarrow \Div_{n,d}(D)$ to $\Div^L_{n,d}(D)_{\underline{m}}^+$ are isomorphic; this follows from the description of the normal bundle in \cite[$\S$6]{BD}.

Since $j_D$ is transverse to the BB strata, Conjecture \ref{thm2} for the transition maps in the motivic BB decompositions of $\Div$ implies the analogous statement for $\Div^L$, as we obtain the analogous morphisms by intersecting with the classes of $\Div^L_{n,d}(D) \times C^{(\underline{m})}_L \hookrightarrow \Div_{n,d}(D) \times C^{(\underline{m})}$. 
\end{proof}

\begin{thm}\label{fixed det}
Assume that Conjecture~\ref{thm2} holds and $C(k) \neq \emptyset$. In $\DM(k,R)$, we have
\[M(\Bun_{n,d}^L)  \simeq  M(B\GG_m) \otimes \bigotimes_{i=1}^{n-1} Z(C, R\{i\}) \]
and
\[M^c(\Bun_{n,d}^L)  \simeq M^c(B\GG_m)\{(n^2 -1 )(g-1) \}  \otimes \bigotimes_{i=2}^n Z(C, R\{-i\}). \]
\end{thm}
\begin{proof}
First one restricts $\Div_{n,d} \ra \Bun_{n,d}$ to $\Div_{n,d}^L \ra \Bun_{n,d}^L$ and, analogously to Theorem \ref{thm1}, one proves that for any non-zero effective divisor $D_0$ on $C$, there is an isomorphism
\[M(\Bun_{n,d}^L) \simeq \hocolim_{l} M(\Div_{n,d}^L(lD_0)) \]
in $\DM(k,R)$; for the relevant codimension estimates, one can use \cite[\S 6]{BD} together with Lemma \ref{codim bgl correct}. We then consider the Bia{\l}ynicki-Birula decompositions for the smooth closed $\GG_m$-invariant subvarieties $\Div_{n,d}^L(lD_0) \subset \Div_{n,d}(lD_0)$, whose fixed loci is the disjoint union of $C^{(\underline{m})}_L \subset C^{(\underline{m})}$ for partitions $\underline{m}$ on $n\deg(D_0)l - d$. By Proposition \ref{conj for fixed det}, the codimensions of the BB strata in $\Div_{n,d}^L(lD_0)$ equal the codimensions of the BB strata in $\Div_{n,d}(lD_0)$ as $j_{lD_0}$ is transverse to the BB strata. 

By the final statement of Proposition \ref{conj for fixed det}, we can follow the argument of the proof of Theorem \ref{thm3} and define $P^L_{\underline{m}^{\flat},l}$ as in the proof of Theorem \ref{thm3}, but by replacing $C^{(\underline{m}(l))}$ by $C^{(\underline{m}(l))}_L$; then analogously to \eqref{arepa} we have a corresponding isomorphism with superscript $L$ inserted on both sides. For $\underline{m}_1(l)>2g-2$, the projection morphism
\[ C^{(\underline{m})}_L \ra C^{(m_2)} \times \cdots \times C^{(m_n)} \]
is a $\PP^{m_1-g}$-bundle (as we have fixed the determinant and $C(k)\neq \emptyset$). Hence
\begin{align*}
  \hocolim_l P^L_{\underline{m}^{\flat},l} &\simeq \hocolim_{l:\ m_1(l)>2g-2}P^L_{\underline{m}^{\flat},l}\simeq \hocolim_{l:\ m_1(l)>2g-2} M(\PP^{m_1(l)-g})\otimes M(C^{(\underline{m}^\flat)}\{c_{\underline{m}^\flat}\} \\ & \simeq M(B\GG_m)\otimes M(C^{(\underline{m}^\flat)})\{c_{\underline{m}^\flat}\}
  \end{align*}
where we have used Lemma \ref{pulling out constants hocolims} and Example \ref{ex mot BGm}. Then the remainder of the proof for the formula for $M(\Bun_{n,d}^L)$ follows that of Theorem \ref{thm3} verbatim, and the formula for $M^c(\Bun_{n,d}^L)$ follows by Poincar\'{e} duality, as the codimension of $\Bun_{n,d}^L$ in $\Bun_{n,d}$ is $g$.
\end{proof}

\begin{thm}\label{SL-bundles}
Assume that Conjecture~\ref{thm2} holds and $C(k) \neq \emptyset$. Then, in $\DM(k,R)$, we have
\[M(\Bun_{\SL_n})  \simeq  \bigotimes_{i=1}^{n-1} Z(C, R\{i\})\]
and
\[M^c(\Bun_{\SL_n})  \simeq  \bigotimes_{i=2}^n Z(C, R\{-i\}) \otimes R \{(n^2 -1 )(g-1) \} . \]
\end{thm}
\begin{proof} 
By Poincar\'e duality and the same argument as in the proof of Theorem \ref{compact supp}, the formula for $M^c$ follows from the formula for $M$.

As $\Bun_{\SL_n} \ra \Bun_{n,0}^{\cO_C}$ is a $\GG_m$-torsor equal to the pullback of the universal $\GG_m$-torsor on $B\GG_m$ via the morphisms $\det : \Bun_{n,0}^{\cO_C} \ra B\GG_m$, the idea is to use Proposition \ref{motives of mult bundles} and Example \ref{ex univ Gm torsor}. Indeed, by Proposition \ref{motives of mult bundles}, we have a distinguished triangle
\begin{equation}\label{bike}
 M(\Bun_{\SL_n}) \longrightarrow M(\Bun_{n,0}^{\cO_C}) \stackrel{\varphi}{\longrightarrow} M(\Bun_{n,0}^{\cO_C})\{1\} \stackrel{+}{\ra}
\end{equation}
where $\varphi:=\varphi_{\cL_n}$ is defined using the line bundle $\cL_n \ra  \Bun_{n,0}^{\cO_C}$ associated to the above $\GG_m$-torsor. In fact, if $\cU_n \ra \Bun_{n,0}^{\cO_C}\times C$ denotes the universal bundle and $\pi_1 : \Bun_{n,0}^{\cO_C} \times C \ra \Bun_{n,0}^{\cO_C}$ denotes the projection, then $\det(\cU_n) \cong \pi_1^*(\cL_n)$ by the See-saw Theorem, as $C$ is compact, and we are working on the stack of bundles with fixed determinant. 

By Theorem \ref{fixed det}, we have $M(\Bun_{n,0}^{\cO_C})  \simeq Z \otimes M(B\GG_m)$, where $Z := \bigotimes_{i=1}^{n-1} Z(C, R\{i\})$. By using the projection $M(B\GG_m) \ra R\{ 0 \}$, we can construct a morphism
\[ s: M(\Bun_{\SL_n} ) \ra M(\Bun_{n,0}^{\cO_C}) \simeq  Z \otimes M(B\GG_m)  \ra Z \otimes R\{ 0 \} \simeq Z, \]
which we claim is an isomorphism. Let $\widetilde{\varphi} : M(B\GG_m) \ra M(B\GG_m)\{1\}$ be the morphism induced by the universal line bundle $\cO(-1)=[\AA^{1}/\GG_{m}]$ on $B\GG_m$ in the sense of Proposition \ref{motives of mult bundles}; this morphism was computed in Example \ref{ex univ Gm torsor}. Then we have the following diagram
\begin{equation}\label{tea}
\xymatrix{ M(\Bun_{\SL_n})  \ar[r] \ar[d]_{s} & M(\Bun_{n,0}^{\cO_C}) \ar[r]^{\varphi}  \ar[d]_{\wr} & M(\Bun_{n,0}^{\cO_C})\{ 1\}  \ar[d]_{\wr}  \ar[r] &\\
Z \ar[r] & Z \otimes M(B\GG_m) \ar[r]^{\id \otimes \widetilde{\varphi}} & Z \otimes M(B\GG_m)\{ 1\}  \ar[r] & }
\end{equation}
where both rows are distinguished triangles and the left hand square commutes by definition of $s$. We claim that the right hand square also commutes, from which it follows that $s$ is an isomorphism.

To prove the claim, by definition of $\varphi$ and $\widetilde{\varphi}$ in terms of first Chern classes in Proposition \ref{motives of mult bundles}, it is enough to show that the map $c_{1}(\cL_{n}):M(\Bun_{n,0}^{\cO_{C}})\ra R\{1\}$ is equal to the composition \[M(\Bun_{n,0}^{\cO_{C}})\simeq Z\otimes M(B\GG_{m})\lra M(B\GG_{m})\stackrel{c_{1}(\cO(-1))}{\longrightarrow}R\{1\}.\]
We will show this via the presentation of $\Bun_{n,0}^{\cO_C}$ in terms of matrix divisors with determinant $\cO_C$ which was used to construct the isomorphism in Theorem \ref{fixed det}.

Fix an effective divisor $D_{0}$ of degree $d_{0}>0$ and let $\cL_{n,l}$ be the pullback of $\cL_{n}$ along $\Div_{n,0}^{\cO_C}(lD_0) \ra \Bun_{n,0}^{\cO_C}$. By definition of Chern classes of exhaustive stacks, we have to prove that, for all $l\in \NN$ large enough, the map $c_{1}(\cL_{n,l}):M(\Div_{n,0}^{\cO_C}(lD_0))\ra R\{1\}$
is equal to the composition
\begin{equation}\label{panda}
M(\Div_{n,0}^{\cO_C}(lD_0))\lra M(\Bun_{n,0}^{\cO_{C}})\simeq Z\otimes M(B\GG_{m})\ra M(B\GG_{m})\stackrel{c_{1}(\cO(-1))}{\longrightarrow}R\{1\}.
\end{equation}
There is a natural determinant morphism $\det : \Div_{n,0}^{\cO_C}(lD_0) \ra  \Div_{1,0}^{\cO_C}(nlD_0)$. Again by the See-saw theorem, we have $\pi_1^*(\cL_{n,l}) \cong \det(\cE_{n,l})$, where $\cE_{n,l} \ra \Div_{n,0}^{\cO_C}(lD_0) \times C$ is the universal bundle on this Quot scheme. We have $(\det \times {\id_C})^*(\cE_{1,nl}) \cong \det (\cE_{n,l})$ and again by the See-saw Theorem, there is a line bundle $\cL_{1,nl} \ra \Div_{1,0}^{\cO_C}(nlD_0)$ such that $\pi_1^*\cL_{1,nl} \cong \cE_{1,nl}$.

We now have to delve into the proof of Theorem \ref{fixed det}. In particular, for $nld_{0}>2g-2$, we see that $\Div_{1,0}^{\cO_C}(nlD_0)$ is isomorphic to $\PP^{nld_{0}-g}$. Moreover, via this isomorphism, $\cL_{1,nl}$ is identified with the tautological bundle $\cO(-1)$ on this projective space. By the proof of Theorem \ref{fixed det}, we have a commutative diagram
\[
  \xymatrix{
    M(\Div_{n,0}^{\cO_C}(lD_0)) \ar[d]_{M(\det)} \ar[r] & M(\Bun_{n,0}^{\cO_{C}}) \ar[r]^{\sim \quad} & Z\otimes M(B\GG_{m}) \ar[d] \\
    M(\Div_{1,0}^{\cO_C}(nlD_0)) \ar[r]^{\quad \: \sim} & M(\PP^{nld_{0}-g})\ar[r] & M(B\GG_{m}).
  }
\]
Hence the composition \eqref{panda} is equal to
\[M(\Div_{n,0}^{\cO_C}(lD_0))\stackrel{M(\det)}{\lra}M(\Div^{\cO_C}_{1,0}(lD_0))\lra M(B\GG_{m})\stackrel{c_{1}(\cO(-1))}{\lra}R\{1\}\]
or in other words to the first Chern class of the pullback of $\cO(-1)$ along $\Div_{n,0}^{\cO_{C}}(lD_{0})\ra B\GG_{m}$. As observed above, this pullback is equal to $\det^{*}(\cL_{1,nl})\simeq \cL_{n,l}$, which concludes the proof.
\end{proof}

\appendix

\section{\'{E}tale motives of stacks}\label{app mot stacks}

Let us explain and compare some definitions of motives of stacks (see also \cite{choudhury_phd,totaro,hoyois}). Let $\DA^{\et}(k,R)$ denote the triangulated category of \'{e}tale motives without transfers over $k$ with coefficients in $R$, a close variant of $\DM(k,R)$ \cite[\S 3]{ayoub_etale}, which is related to $\DM(k,R)$ through two adjunctions
\[
\LL a_{\tr}:\DA^{\et}(k,R)\leftrightarrows \DM^{\et}(k,R):\RR o_{\tr}
\]
and
\[
\LL a_{\et} : \DM^{\Nis}(k,R) \leftrightarrows \DM^{\et}(k,R): \RR o_{\et}.
\]
With our standing hypotheses on $(k,R)$, the functor $\LL a_{\tr}$ is always an equivalence \cite[Theorem B.1]{ayoub_etale}, while the functor $\LL a_{\et}$ is an equivalence if $R$ is a $\QQ$-algebra \cite[Theorem 14.30]{MVW}.

In $\DA^{\et}(k,R)$, there are two natural ways to define the motive of an algebraic stack $\fX$: first using the \v{C}ech hypercover associated to an smooth atlas and second via a nerve construction.

\begin{defn}\label{def atlas nerve}
Let $\fX$ be an algebraic stack. We define the following motives.
\begin{enumerate}[label={\upshape(\roman*)}]
\item For a smooth atlas $\cA \ra \fX$, we let $\cA^{[n]}:=\cA \times_{\fX}  \cdots \times_{\fX} \cA$ be the $n+1$-fold self-fibre product of $\cA$ in $\fX$ or equivalently the $n$th \v{C}ech simplicial scheme associated to this atlas\footnote{$\cA^{[n]}$ is a scheme itself, as $\cA$ is assumed to be a scheme and the morphism $\cA \ra \fX$ is representable}. We define the \'{e}tale motive of $\fX$ with respect to this atlas as
  \[ M^{\et}_{\atlas,\cA}(\fX):= \Sigma^{\infty}_{T}R^{\eff}(\cA^{[\bullet]}). \]
where $R^{\eff}(\cA^{[\bullet]})$ denotes the complex of \'{e}tale sheaves associated to the simplicial sheaf of $R$-modules $\cA^{[\bullet]}$.

\item Let $f : \fX \ra \cY$ be a representable morphism of stacks; then for any atlas $\cA_\cY \ra \cY$, the map $\cA_\fX := \cA_\cY \times_{\cY} \fX$ is an atlas for $\fX$ and thus there is a morphism
\[ M^{\et}_{\atlas}(f) : M^{\et}_{\atlas, \cA_\fX}(\fX) \ra  M^{\et}_{\atlas, \cA_\cY}(\cY).\]  
\item We define the nerve-theoretic motive of $\fX$ by taking the functor $\fX : \Sm_k \ra \Gpds$ valued in the category of groupoids and composing with the nerve $N: \Gpds \ra \sSets$ to the category of simplicial sets and then finally composing with the singular complex $\sing : \sSets \ra C_*(\abgp)$ which takes values in the category of complexes of abelian groups; this gives a complex of \'{e}tale sheaves of abelian groups $a_{\et}\sing \circ N\circ \fX$ on $\Sm_k$, which we tensor with our coefficient ring $R$ and define
\[ M^{\et}_{\nerve}(\fX):= \Sigma^{\infty}_{T}(a_{\et}(\sing \circ N\circ \fX) \otimes R).\]
\end{enumerate}
\end{defn} 

\begin{rmk}\
\begin{enumerate}[label={\upshape(\roman*)}]
\item By definition, the atlas and nerve theoretic definitions of the motive of an algebraic stack are both effective motives, as both definitions take place in $\DA^{\eff,\et}(k,R)$. 
\item If $\fX$ is a scheme $X$, then clearly we have $M(X) \simeq M_{\text{atlas},X}(X)$ and in fact this also agrees with the nerve-theoretic definition by Lemma \ref{lemma atlas and nerve} below.
\end{enumerate}
\end{rmk}

\begin{lemma}\label{lemma atlas and nerve}
In $\DA^{\et}(k,R)$ for an algebraic stack $\fX$ and for any atlas $\cA \ra \fX$, we have
\[M^{\et}_{\atlas,\cA}(\fX) \simeq M^{\et}_{\nerve}(\fX).\]
In particular, the isomorphism class of $M_{\atlas,\cA}(\fX)$ is independent of the atlas.
\end{lemma}
\begin{proof}
This holds as $\DA^{\et}(-,R)$ satisfies cohomological descent with respect to $h$-hypercoverings \cite[Theorem 14.3.4.(2)]{cd}, thus in particular with respect to \v{C}ech hypercoverings induced by smooth surjective maps. For a more detailed argument, see \cite{choudhury_phd} which treats the case of a Deligne-Mumford stack $\fX$ in $\DM^{\et}(k,\QQ)$ (replacing \'etale descent by smooth descent).
\end{proof}

For an exhaustive stack $\fX$, one can define the motive of $\fX$ in $\DA^{\et}(k,R)$ (or in fact, for any choice of topology) exactly as in Definition \ref{def mot stack} by using an exhaustive sequence of vector bundles with respect to a filtration of $\fX$. In this appendix, to distinguish the motive in Definition \ref{def mot stack} (resp.\ its \'{e}tale, without transfers analogue) from other definitions, we denote it by $M_{\mono}(\fX)$ (resp.\ $M^{\et}_{\mono}(\fX)$) and say $\fX$ is mono-exhaustive (rather than just exhaustive).

In fact, dual to Definition \ref{def exh seq} of an (injective) exhaustive sequence is the following notion of a surjective exhaustive sequence as in \cite[$\S$8]{totaro}.

\begin{defn}\label{def surj exh seq}
For a smooth stack $\fX$, a surjective exhaustive sequence of vector bundles on $\fX$ with respect to a filtration $\fX_0\stackrel{i_0}{\ra} \fX_1\stackrel{i_1}{\ra}\ldots\subset \fX$ is a pair $(V_\bullet, W_\bullet)$ given by a sequence of vector bundles $V_m$ over $\fX_m$ together with surjective maps of vector bundles $f_m: V_{m+1}\times_{\fX_{m+1}}\fX_m  \ra V_m$ and closed substacks $W_m\subset V_m$ such that
\begin{enumerate}[label={\upshape(\roman*)}]
\item  the codimension of $W_m$  in $V_m$ tend towards infinity,
\item the complement $U_m:= V_m - W_m$ is a separated finite type $k$-scheme, and
\item we have $W_{m+1}\times_{\fX_{m+1}}\fX_m  \subset f_m^{-1}(W_{m})$.
\end{enumerate}
If $\fX$ admits such a surjective sequence, we say $\fX$ is a epi-exhaustive stack. Then we define the motive of $\fX$ (resp.\ the \'{e}tale motive of $\fX$) with respect to such a surjective exhaustive sequence $(V_\bullet, W_\bullet)$ of vector bundles on $\fX$ by
\[ M_{\epi}(\fX):= \hocolim_m M(U_m)\in \DM(k,R)\quad\quad \text{(resp.}\ M^{\et}_{\epi}(\fX):= \hocolim_m M(U_m)\in \DM^{\et}(k,R)))\]
with transition maps given by the composition
\[ M(U_m) \cong M((V_{m+1}\times_{\fX_{m+1}}\fX_m)  - f_m^{-1}(W_m)) \ra M(U_{m+1}\times_{\fX_{m+1}}\fX_m ) \ra M(U_{m+1}) \]
where the first isomorphism follows as $(V_{m+1}\times_{\fX_{m+1}}\fX_m)  - f_m^{-1}(W_m) \ra (V_m - W_m) = U_m$ is a Zariski locally trivial affine space fibration, and the final maps come from the open immersions $V_{m+1}\times_{\fX_{m+1}}\fX_m  - f_m^{-1}(W_m) \hookrightarrow U_{m+1}\times_{\fX_{m+1}}\fX_m \hookrightarrow U_{m+1}$ (resp.\ the analoguous maps for \'etale motives). Similarly, we define the compactly supported motive of $\fX$ with respect to this sequence by
\[ M^c_{\epi}(\fX):= \holim_m M(U_m)\{ - r_m\} \]
where $r_m := \rk(V_m)$, with transition maps given by the composition
\[\xymatrix@R=0.5cm{
M^c(U_{m+1})\{ -r_{m+1}\} \ar[r] &  M^c(U_{m+1}\times_{\fX_{m+1}}\fX_m )\{ -r_{m+1}\} \ar[d] \\
 & M^c(V_{m+1}\times_{\fX_{m+1}}\fX_m  - f_m^{-1}(W_m))\{ -r_{m+1}\} & \ar[l]_{\quad \quad\quad \quad \quad \quad \quad\simeq} M^c(U_m) \{ - r_m\}
 }\]
where the first two morphisms are flat pullbacks associated to open immersions and the final isomorphism follows from $\AA^1$-homotopy invariance. 
\end{defn}

\begin{rmk}
Unlike in Definition \ref{def mot stack}, since we have homotopy (co)limits, Definition \ref{def surj exh seq} depends on the choice of a cone, hence is only defined up to a non-unique isomorphism.
\end{rmk}

Totaro proves his definition of the compactly supported motive of a quotient stack is independent of the choice of such sequence \cite[Theorem 8.5]{totaro}; in fact, his proof can also be adapted following the lines of Lemma \ref{lem stack def indept} to show the above definitions for the motive of a epi-exhaustive smooth stack are independent of the above choices. 

\begin{rmk}\label{rem tr et}
For a mono-exhaustive (resp.\ epi-exhaustive) stack $\fX$, this construction immediately implies the following isomorphisms in $\DM^{\et}(k,R)$
  \[\LL a_{\tr} M^{\et}_{\mono}(\fX)\simeq \LL a_{\et} M_{\mono}(\fX)\quad\quad\text{(resp.}\ \LL a_{\tr} M^{\et}_{\epi}(\fX)\simeq \LL a_{\et} M_{\epi}(\fX)).\]
\end{rmk}

\begin{prop}\label{prop atlas and our defn}
Let $\fX$ be a mono-exhaustive (resp.\ epi-exhaustive) smooth stack; then we have
\[  \LL a_{\tr} M^{\et}_{\atlas,\cA}(\fX) \simeq \LL a_{\et} M_{\mono}(\fX)\quad\quad\text{(resp.}\ \LL a_{\tr} M^{\et}_{\atlas,\cA}(\fX) \simeq \LL a_{\et} M_{\epi}(\fX))\]
in $\DM^{\et}(k,R)$ for any choice of atlas $\cA \ra \fX$. In particular, when $R$ is a $\QQ$-algebra, the motive $M_{\mono}(\fX)$ (resp.\ $M_{\epi}(\fX)$) can be computed from the \'{e}tale motive.
\end{prop}
\begin{proof}
We present the proof in the mono-exhaustive case, the epi-exhaustive case is similar. By the first isomorphism in Remark~\ref{rem tr et} and the fact that $\LL a_{\tr}$ is an equivalence of categories, we see that it suffices to construct an isomorphism $M^{\et}_{\atlas,\cA}(\fX) \simeq M^{\et}_{\mono}(\fX)$.

Let $\cA\ra \fX$ be any atlas of $\fX$. In particular, $\cA$ is a smooth k-scheme, not of finite type in general. Recall that by hypothesis, there is a filtration $\fX= \bigcup_{m\in \NN} \fX_m$ by increasing quasi-compact open substacks and an exhaustive sequence of vector bundles $(V_\bullet,W_\bullet)$ with respect to this filtration.
For any $m\in \NN$, because $\fX_{m}$ is quasi-compact there exists a union $\cA_{m}$ of finitely many connected components of $\cA\times_{\fX}\fX_{m}$  (so that $\cA_{m}$ is a smooth finite type $k$-scheme) such that $\cA_{m}\ra \fX_{m}$ is an atlas; we can assume furthermore that $\cA_{m}\subset \cA_{m+1}$ and that $\cA=\cup_{m}\cA_{m}$. This implies that the \'{e}tale sheaf $R^{\eff}(\cA)$ is the colimit of $(R^{\eff}(\cA_{m}))_{m\in \NN}$. For $m,k\geq 0$, we let 
\[ \tau_{\leq k}(R(\cA_m^{[\bullet]})):= \Sigma^{\infty}_{T}\left( \quad  \cdots\ra 0 \ra  R(\cA_m^{[k-1]}) \ra  \cdots \ra R(\cA_m)\ra 0 \ra \cdots \quad \right). \]
In the category of $T$-spectra of complexes of \'{e}tale sheaves, we have
\begin{align*}
R(\cA^{[\bullet]}) \simeq &  \colim_k  \tau_{\leq k}(R(\cA^{[\bullet]}))  \simeq \colim_k \colim_m \tau_{\leq k}(R(\cA_m^{[\bullet]})) \\
\simeq & \colim_m \colim_k  \tau_{\leq k}(R(\cA_m^{[\bullet]})) \simeq \colim_m R(\cA_{m}^{[\bullet]}).
\end{align*}
For $m\in\NN$, let $\widetilde{V}_{m}:=V_{m}\times_{\fX_{m}} \cA_{m}$ and define $\widetilde{W}_{m}$ and $\widetilde{U}_{m}$ analogously; as $\cA\ra \fX$ is representable, these are (smooth, finite type) $k$-schemes. The maps $V_{m}\ra V_{m+1}$ lift to maps $\widetilde{V}_{m}\ra \widetilde{V}_{m+1}$. We have
\[
\colim_m R(V_{m}\times_{\fX}\cA_{m}^{[\bullet]})\simeq\colim_{m} R(\cA_{m}^{[\bullet]})
  \]
  which is a filtered colimit of $\AA^{1}$-equivalences, thus an $\AA^{1}$-equivalence. We have $V_{m}\times_{\fX}\cA_{m}^{[n]}\simeq \widetilde{V}_{m}\times_{k}\cA_{m}^{[n-1]}$ for all $n\geq 1$, thus each term is a smooth finite type $k$-scheme and
  \[
\codim_{(V_{m}\times_{\fX}\cA_{m}^{[n]})}(W_{m}\times_{\fX}\cA_{m}^{[n]})\geq \codim_{V_m}W_m
\]
tends to $\infty$ uniformly in $n$. Proposition \ref{prop hocolim vanishes} and a truncation argument as above imply that
  \[
\colim_m R(U_{m}\times_{\fX}\cA_{m}^{[\bullet]}) \ra \colim_m R(V_{m}\times_{\fX}\cA_{m}^{[\bullet]})
    \]
is an $\AA^{1}$-equivalence. For every $m\geq 1$, the morphism $R(U_{m}\times_{\fX} \cA^{[\bullet]}_{m})\ra R(U_{m})$ is an $\AA^{1}$-equivalence by cohomological descent with respect to $h$-hypercoverings \cite[Theorem 14.3.4.(2)]{cd}. We finally conclude that the map
\[
\colim_m R(V_{m}\times_{\fX}\cA_{m}^{[\bullet]})\ra \colim_{m}R(U_{m})
\]
is an $\AA^{1}$-equivalence since it is a filtered colimit of $\AA^{1}$-equivalences. Putting everything together, we get an isomorphism in $\DA^{\et}(k,R)$
\[
M^{\et}_{\atlas,\cA}(\fX)\simeq \colim_{m}R(U_{m})=:M^{\et}_{\mono}(\fX)
  \]
which concludes the proof.\end{proof}

\bibliographystyle{amsplain}
\bibliography{references}

\medskip \medskip

\noindent{Freie Universit\"{a}t Berlin, Arnimallee 3, 14195 Berlin, Germany} 

\medskip \noindent{\texttt{hoskins@math.fu-berlin.de, simon.pepin.lehalleur@gmail.com}}

\end{document}